\documentclass[12pt]{amsart}%
\usepackage{amsfonts}
\usepackage{amsmath}
\usepackage{graphicx}
\usepackage{amssymb}%
\providecommand{\U}[1]{\protect\rule{.1in}{.1in}}
\newtheorem{theorem}{Theorem}

\newtheorem{example}[theorem]{Example}

\newtheorem{lemma}[theorem]{Lemma}

\newtheorem{proposition}[theorem]{Proposition}
\newtheorem{remark}[theorem]{Remark}

\begin{document}
\title{Polyharmonic functions of infinite order \\on annular regions}
\author[O. Kounchev and H. Render]{Ognyan Kounchev and Hermann Render}
\address{O. Kounchev: Institute of Mathematics and Informatics, Bulgarian Academy of
Sciences, 8 Acad. G. Bonchev Str., 1113 Sofia, Bulgaria and IZKS, University
of Bonn.\\
H. Render: School of Mathematical Sciences, University College Dublin, Dublin
4, Ireland.}
\email{hermann.render@ucd.ie}
\thanks{The first author was partially supported by the Alexander von Humboldt
Foundation. The second author is partially supported by Grant
MTM2009-12740-C03-03 of the D.G.I. of Spain. Both authors acknowledge support
within the Project DO-02-275 Astroinformatics with the NSF of Bulgarian
Ministry of Education and Science.}

\begin{abstract}
Polyharmonic functions $f$ of infinite order and type $\tau$ on annular
regions are systematically studied. The first main result states that the
Fourier-Laplace coefficients $f_{k,l}\left(  r\right)  $ of a polyharmonic
function $f$ of infinite order and type $0$ can be extended to analytic
functions on the complex plane cut along the negative semiaxis. The second
main result gives a constructive procedure via Fourier-Laplace series for the
analytic extension of a polyharmonic function on annular region $A\left(
r_{0},r_{1}\right)  $ of infinite order and type less than $1/2r_{1}$ to the
kernel of the harmonicity hull of the annular region. The methods of proof
depend on an extensive investigation of Taylor series with respect to linear
differential operators with constant coefficients.

\end{abstract}
\maketitle

\section{Introduction}

Let $G$ be a domain in the euclidean space $\mathbb{R}^{d}.$ A function
$f:G\rightarrow\mathbb{C}$ is called polyharmonic of order $p$ if $f$ is $2p$
times continuously differentiable and $\Delta^{p}f\left(  x\right)  =0$ for
all $x\in G,$ where $\Delta=\partial^{2}/\partial x_{1}^{2}+\cdots
+\partial^{2}/\partial x_{d}^{2}$ is the Laplace operator and $\Delta^{p}$ the
$p$-th iterate of $\Delta.$ Polyharmonic functions have been investigated by
several authors (see e.g.\ \cite{Alma99}, \cite{Eden}, \cite{EnPe96},
\cite{FKM01}, \cite{FKM03}, \cite{Hayman}, \cite{kounchev92},
\cite{kounchev98}, \cite{Ligo88}, \cite{Nico35}, \cite{Render08},
\cite{Render08b}, \cite{Sob}) and they have recently many applications in
approximation theory, radial basis functions and wavelet analysis (see
e.g.\ \cite{BBRV05}, \cite{Koun00}, \cite{KoReCM}, \cite{kounchevrenderJAT},
\cite{MaNe90}).

Aronszaijn introduced in 1935 the concept of a polyharmonic function of
\emph{infinite order} (see \cite{Aron35} and \cite{Lelo46}). On the one hand,
this class of functions contains the class of classical polyharmonic functions
of all finite orders $p,$ and on the other hand it retains many properties of
the latter class, e.g.\ analytic extendibility to the harmonicity hull; the
monograph \cite{ACL83} is devoted to this subject and additional information
can be found in the research book of Avanissian \cite{Avan85}. Important
examples are eigenfunctions of the Laplacian, i.e., functions satisfying the
equation $\Delta f\left(  x\right)  =\lambda f\left(  x\right)  $ for some
$\lambda\in\mathbb{C},$ or so-called \emph{metaharmonic} functions
(e.g.\ \cite{Veku43}).

Let us recall that a function $f:G\rightarrow\mathbb{C}$ is \emph{polyharmonic
of infinite order and type} $\tau\geq0$ if, for any compact set $K\subset G$
and for all $\varepsilon>0,$ there exists a constant $C_{K,\varepsilon}>0$
such that
\begin{equation}
\max_{x\in K}\left\vert \Delta^{p}f\left(  x\right)  \right\vert \leq
C_{K,\varepsilon}\left(  2p\right)  !\left(  \tau+\varepsilon\right)
^{2p}\text{ } \label{eqpolyinf}%
\end{equation}
for all natural numbers $p.$ An equivalent way to express this inequality is
to require that for any compact subset $K$ of $G$ the inequality
\begin{equation}
\overline{\lim_{p\rightarrow\infty}}\max_{x\in K}\sqrt[2p]{\frac{\left\vert
\Delta^{p}f\left(  x\right)  \right\vert }{\left(  2p\right)  !}}\leq
\tau\label{eqlims}%
\end{equation}
holds. \cite[Theorem 1.4]{Avan85} characterizes real-analyticity in terms of
estimates of the Laplacian. Namely, an infinitely differentiable function
$f:G\rightarrow\mathbb{C}$ is real-analytic if and only if for any compact
subset $K$ of $G$ there exists a constant $C_{K}$ and a constant $\tau_{K}$
such that
\begin{equation}
\max_{x\in K}\left\vert \Delta^{p}f\left(  x\right)  \right\vert \leq
C_{K}\left(  2p\right)  !\left(  \tau_{K}\right)  ^{2p} \label{realana}%
\end{equation}
for all natural numbers $p.$ Thus polyharmonic functions of infinite order and
type $\tau$ are real-analytic and they allow the explicit control of the
constant $\tau_{K}$ in (\ref{realana}).

In the present paper we shall study polyharmonic functions of infinite order
on the annular region
\[
A\left(  r_{0},r_{1}\right)  :=\left\{  x\in\mathbb{R}^{d};r_{0}<\left\vert
x\right\vert <r_{1}\right\}  \text{ for }0\leq r_{0}<r_{1}\leq\infty.
\]
In this case tools from harmonic analysis, like the Fourier-Laplace series,
are available. Our first goal is to describe properties of the Fourier-Laplace
coefficients $f_{k,l}$ of a polyharmonic function $f$ of infinite order. Let
us recall some basic notations: Let
\[
Y_{k,l}\left(  x\right)  ,\qquad\text{for }l=1,..,a_{k},
\]
be an orthonormal basis of the $a_{k}$-dimensional linear space of harmonic
homogeneous polynomials of degree $k\geq0,$ which are orthonormal with respect
to the scalar product
\[
\left\langle f,g\right\rangle _{\mathbb{S}^{d-1}}:=\int_{\mathbb{S}^{d-1}%
}f\left(  \theta\right)  \overline{g\left(  \theta\right)  }d\theta,
\]
where $\mathbb{S}^{d-1}=\left\{  x\in\mathbb{R}^{d};\left\vert x\right\vert
=1\right\}  $ is the unit sphere (see \cite{ABR92}, \cite{Koun00},
\cite{Mull66}, \cite{StWe71}). Let $f$ be a continuous function on the annular
region $A\left(  r_{0},r_{1}\right)  .$ Then the \emph{Fourier-Laplace
coefficients} $f_{k,l}$ of $f$ are defined by
\begin{equation}
f_{k,l}\left(  r\right)  =\int_{\mathbb{S}^{d-1}}f\left(  r\theta\right)
\overline{Y_{k,l}\left(  \theta\right)  }d\theta\text{ for }r\in\left(
r_{0},r_{1}\right)  . \label{LFK}%
\end{equation}
The formal series
\begin{equation}
\sum_{k=0}^{\infty}\sum_{l=1}^{a_{k}}f_{k,l}\left(  r\right)  Y_{k,l}\left(
\theta\right)  \label{LF}%
\end{equation}
is called the \emph{Fourier-Laplace series} of $f.$ The special case of a
harmonic function $f$ defined on $A\left(  r_{0},r_{1}\right)  $ shall serve
us as a guiding example: We have
\begin{equation}
f_{k,l}\left(  r\right)  =\left\{
\begin{array}
[c]{ll}%
\alpha_{k}r^{k}+\beta_{k}r^{-k-d+2} & \text{for }d>2,\text{ or }d=2,k\geq1\\
\alpha_{0}+\beta_{0}\log r & \text{for }d=2,k=0.
\end{array}
\right.  \label{eqharmonic}%
\end{equation}
on the open interval $\left(  r_{0},r_{1}\right)  ,$ for suitable complex
coefficients $\alpha_{k}$ and $\beta_{k}.$ More generally, it is known that if
$f$ is polyharmonic of order $p$ and $d$ is odd then there exist polynomials
$p_{k,l}$ and $q_{k,l}$ of degree $p-1$ such that $f_{k,l}\left(  r\right)
=r^{k}p_{k,l}\left(  r^{2}\right)  +r^{-k-d+2}q_{k,l}\left(  r^{2}\right)  $
(see \cite{Sob}, \cite{Veku67} or \cite{Koun00}). Thus for odd dimension
$f_{k,l}\left(  r\right)  $ extends to an analytic function on the punctured
plane $\mathbb{C}^{\ast}:=\left\{  z\in\mathbb{C};z\neq0\right\}  ,$ while for
even dimension we can only infer that $f_{k,l}\left(  r\right)  $ is an
analytic function on the cutted complex plane $\mathbb{C}\setminus\left(
-\infty,0\right]  .$

The \textbf{first main result} of this paper states the following: The
Fourier-Laplace coefficients $f_{k,l}\left(  r\right)  $ of a polyharmonic
function $f:A\left(  r_{0},r_{1}\right)  \rightarrow\mathbb{C}$ of infinite
order and type $0$ possess analytic extensions to the cutted complex plane
$\mathbb{C}\setminus\left(  -\infty,0\right]  $ (cf. Theorem \ref{ThmMain}
below). For odd dimension we can sharpen the result: There exist \emph{entire
functions} $p_{k,l}$ and $q_{k,l}$ such that
\begin{equation}
f_{k,l}\left(  r\right)  =r^{k}p_{k,l}\left(  r^{2}\right)  +r^{-k-d+2}%
q_{k,l}\left(  r^{2}\right)  \label{ENEW}%
\end{equation}
for all $r\in\left(  r_{0},r_{1}\right)  .$ In particular, it follows that the
Fourier-Laplace coefficients $f_{k,l}$ defined on the interval $\left(
r_{0},r_{1}\right)  $ can be analytically extended to the punctured plane
$\mathbb{C}^{\ast}:=\left\{  z\in\mathbb{C};z\neq0\right\}  .$ We refer to
Theorem \ref{ThmCoeff} below.

The \textbf{second main} \textbf{result} of the paper addresses the problem of
extending analytically a polyharmonic function $f:A\left(  r_{0},r_{1}\right)
\rightarrow\mathbb{C}$ of infinite order and type $\tau\geq0$ to a suitable
domain in $\mathbb{C}^{d}.$ It is well known that polyharmonic functions of
infinite order and type $\tau=0$ defined on a domain $G$ in $\mathbb{R}^{d}$
can be extended analytically to the so-called kernel $\widetilde{G}$ of the
harmonicity hull $\widehat{G}$ (see \cite{ACL83}, \cite{Avan85}), or for a
generalization \cite{Eben}. In the case of the annular region $A\left(
r_{0},r_{1}\right)  $ we can give an explicit formula for the analytic
extension via Fourier-Laplace series and, as a by-product, we show that it
suffices to assume that the function $f$ is polyharmonic of infinite order and
type $\tau<1/2r_{1}$ instead of the stronger assumption of type $\tau=0.$ We
refer to Theorem \ref{ThmMM} and Theorem \ref{ThmMM2} below.

In order to make the results more precise we recall some basic notations in
complex analysis in several variables: For $z=\left(  z_{1},\ldots
,z_{d}\right)  \in\mathbb{C}^{d}$ define $\left\vert z\right\vert
_{\mathbb{C}^{d}}^{2}=\left\vert z_{1}\right\vert ^{2}+\cdots+\left\vert
z_{d}\right\vert ^{2}$ and $q\left(  z\right)  :=z_{1}^{2}+\cdots+z_{d}^{2}$.
The upper and lower \emph{Lie-norm} $L_{+}:\mathbb{C}^{d}\rightarrow\left[
0,\infty\right)  $ and $L_{-}:\mathbb{C}^{d}\rightarrow\left[  0,\infty
\right)  $ are defined by
\[
L_{\pm}\left(  z\right)  =\sqrt{\left\vert z\right\vert _{\mathbb{C}^{d}}%
^{2}\pm\sqrt{\left\vert z\right\vert _{\mathbb{C}^{d}}^{4}-\left\vert q\left(
z\right)  \right\vert ^{2}}}.
\]
The Lie-ball of radius $R\in\left(  0,\infty\right]  $ is defined by
$\widehat{B_{R}}:=\{z\in\mathbb{C}^{d};L_{+}\left(  z\right)  <R\}$ and it is
also called the classical domain of E. Cartan of the type IV, we refer to
\cite[p. 59]{ACL83}, \cite{Hua} or \cite{Mori98} for further details.

In the above terms our \emph{second main result} says that a polyharmonic
function $f:A\left(  r_{0},r_{1}\right)  \rightarrow\mathbb{C}$ of infinite
order and type $\tau<1/2r_{1}$ can be extended to an analytic function on the
domain
\[
\{z\in\mathbb{C}^{d};r_{0}<L_{-}\left(  z\right)  \leq L_{+}\left(  z\right)
<r_{1}\}\setminus q^{-1}\left(  \left(  -\infty,0\right]  \right)  .
\]
The proof depends on a Laurent type decomposition of the function $f$: for odd
dimension $d>1$ we show that there exists an analytic function $f_{1}$ defined
on $\left\{  z\in\mathbb{C}^{d};L_{+}\left(  z\right)  <r_{1}\right\}  $ and
an analytic function $f_{2}$ defined on $\{z\in\mathbb{C}^{d};r_{0}%
<L_{-}\left(  z\right)  \leq L_{+}\left(  z\right)  <1/2\tau\}$ such that the
function $F$ defined by
\[
F\left(  z\right)  =f_{1}\left(  z\right)  +\left(  z_{1}^{2}+\ldots+z_{d}%
^{2}\right)  ^{\left(  2-d\right)  /2}f_{2}\left(  z\right)  \text{ }%
\]
is an analytic extension of $f.$ A similar result is formulated in Section
\ref{S7} for even dimension.

The paper is organized as follows: In the Section $2$ we recall some basic
facts about the action of the Laplace operator $\Delta$ on Fourier-Laplace
series. For polyharmonic functions of infinite order we obtain estimates of
derivatives of the Fourier-Laplace coefficients with respect to certain linear
differential operators depending on the radius $r.$ The results in Section $3$
belong to the main technical merits of the paper: They are devoted to an
extensive discussion of the so-called fundamental function of a linear
differential operator with constant coefficients and the concept of a
generalized Taylor series with respect to the corresponding fundamental
functions, the climax being Theorem \ref{ThmTaylor2} below. These results are
crucial for the main goals of the paper and are also of independent interest.

Section 4 contains the first main result about the analytic extendibility of
the Fourier-Laplace coefficients for polyharmonic functions of infinite order
and type $\tau<1/2r_{0},$ and Section 5 discusses the special case of odd
dimension. In Section 6 we discuss the analytic extendibility as described
above for odd dimension, and in Section 7 the case of even dimension is
addressed. The paper concludes with an Appendix concerning estimates of
generalized derivatives of odd order by even orders in the framework of linear
differential operators with constant coefficients.

Throughout the paper it is assumed that $d\geq2.$ By $\omega_{d-1}$ we define
the surface area of $\mathbb{S}^{d-1}$ with respect to the rotation invariant
measure $d\theta.$

\section{Basic estimates and examples}

By $C^{m}\left(  G\right)  $ we denote the set of all functions
$f:G\rightarrow\mathbb{C}$ which are continuously differentiable up to the
order $m$. It is well known that the Fourier-Laplace series (\ref{LF})
converges absolutely and compactly in $A\left(  r_{0},r_{1}\right)  $ to $f$
if $f\in C^{m}\left(  A\left(  r_{0},r_{1}\right)  \right)  $ for $m>\frac
{1}{2}\left(  d-1\right)  $. We refer to \cite{Kalf95} for questions of
convergence of Fourier-Laplace series and the cited literature therein.

Let $f\in C^{\infty}\left(  A\left(  r_{0},r_{1}\right)  \right)  $ and let
$f_{k,l},k\in\mathbb{N}_{0},l=1,\ldots,a_{k},$ be the Fourier-Laplace
coefficients defined in (\ref{LFK}). Recall that $a_{k}$ is the dimension of
the space of all harmonic homogeneous polynomials of degree $k.$

For $x\in A\left(  r_{0},r_{1}\right)  $ we use spherical coordinates
$x=r\theta$ where $\theta=x/\left\vert x\right\vert $ and $r=\left\vert
x\right\vert .$ It is not difficult to establish the formula
\begin{equation}
\left(  \Delta^{p}f\right)  \left(  r\theta\right)  =\sum_{k=0}^{\infty}%
\sum_{l=1}^{a_{k}}L_{k}^{p}\left(  f_{k,l}\right)  \left(  r\right)  \cdot
Y_{k,l}\left(  \theta\right)  , \label{eqDeltap}%
\end{equation}
where the series converges absolutely and uniformly on compact subsets of
$A\left(  r_{0},r_{1}\right)  $ and $L_{k}^{p}$ is the $p$-th iterate of the
differential operator
\begin{equation}
L_{k}=\frac{d^{2}}{dr^{2}}+\frac{d-1}{r}\frac{d}{dr}-\frac{k\left(
k+d-2\right)  }{r^{2}} \label{deflk}%
\end{equation}
(see \cite{Koun00}). It follows that the $\left(  k,l\right)  $-th
Fourier-Laplace coefficient of $\Delta^{p}f$ is equal to $L_{k}^{p}%
f_{k,l}\left(  r\right)  ,$ i.e., that
\begin{equation}
L_{k}^{p}\left(  f_{k,l}\right)  \left(  r\right)  =\int_{\mathbb{S}^{d-1}%
}\left(  \Delta^{p}f\right)  \left(  r\theta\right)  \overline{Y_{k,l}\left(
\theta\right)  }d\theta\label{eqLLLp}%
\end{equation}
for any $r\in\left(  r_{0},r_{1}\right)  .$ Parseval's formula yields
\[
\int_{\mathbb{S}^{d-1}}\left\vert \Delta^{p}f\left(  r\theta\right)
\right\vert ^{2}d\theta=\sum_{k=0}^{\infty}\sum_{l=1}^{a_{k}}\left\vert
L_{k}^{p}f_{k,l}\left(  r\right)  \right\vert ^{2}<\infty
\]
for any $r$ with $r_{0}<r<r_{1}.$

\begin{theorem}
\label{ThmLkpest}Let $f\in A\left(  r_{0},r_{1}\right)  \rightarrow\mathbb{C}$
be polyharmonic of infinite order and type $\tau\geq0.$ Then for each
subinterval $\left[  a,b\right]  $ of $\left(  r_{0},r_{1}\right)  $ and for
all $\varepsilon>0$ there exists a positive number $C_{a,b,\varepsilon}$ such
that, for all $k\in\mathbb{N}_{0},$ $l=1,\ldots,a_{k},$ for all $r\in\left[
a,b\right]  $ and $p\in\mathbb{N}_{0},$
\[
\left\vert L_{k}^{p}\left(  f_{k,l}\right)  \left(  r\right)  \right\vert \leq
C_{a,b,\varepsilon}\sqrt{\omega_{d-1}}\left(  2p\right)  !\left(
\tau+\varepsilon\right)  ^{2p}.
\]

\end{theorem}

\begin{proof}
Let $\left[  a,b\right]  \subset\left(  r_{0},r_{1}\right)  $ and $K\left(
a,b\right)  :=\left\{  x\in\mathbb{R}^{d};a\leq\left\vert x\right\vert \leq
b\right\}  .$ Then (\ref{eqLLLp}) implies that
\begin{equation}
\left\vert L_{k}^{p}\left(  f_{k,l}\right)  \left(  r\right)  \right\vert
\leq\max_{x\in K\left(  a,b\right)  }\left\vert \Delta^{p}f\left(  x\right)
\right\vert \cdot\int_{\mathbb{S}^{d-1}}\left\vert Y_{k,l}\left(
\theta\right)  \right\vert d\theta\label{eqlast}%
\end{equation}
for all $r\in\left[  a,b\right]  .$ The integral on the right-hand side in
(\ref{eqlast}) can be estimated by the Cauchy Schwarz inequality
\[
\int_{\mathbb{S}^{d-1}}\left\vert Y_{k,l}\left(  \theta\right)  \right\vert
d\theta\leq\sqrt{\int_{\mathbb{S}^{d-1}}1d\theta}\sqrt{\int_{\mathbb{S}^{d-1}%
}\left\vert Y_{k,l}\left(  \theta\right)  \right\vert ^{2}d\theta}%
=\sqrt{\omega_{d-1}}.
\]
Now the result follows from the definition of a polyharmonic function of
infinite order and type $\tau\geq0$ given in (\ref{eqpolyinf}).
\end{proof}

The rest of this Section is devoted to an instructive example: For a real
number $\alpha$ and a harmonic homogeneous polynomial $Y_{k}\left(  x\right)
$ of degree $k\in\mathbb{N}_{0},$ define the function
\[
H_{\alpha,k}\left(  x\right)  =\left(  x_{1}^{2}+\ldots+x_{d}^{2}\right)
^{\alpha}\cdot Y_{k}\left(  x\right)  =\left\vert x\right\vert ^{2\alpha}\cdot
Y_{k}\left(  x\right)  .
\]
Obviously $H_{\alpha,k}$ can be defined on the annular region $\mathbb{R}%
^{d}\setminus\left\{  0\right\}  $. In the next result we restrict
$H_{\alpha,k}$ to the annular region $A\left(  r_{0},r_{1}\right)  $ with
$r_{0}>0$ and we shall show that $H_{\alpha,k}$ is polyharmonic of infinite
order and type at most $1/r_{0}.$

\begin{theorem}
For $\alpha\in\mathbb{N}_{0}$ or $\alpha=1-\frac{1}{2}d-k+j$ with
$j\in\mathbb{N}_{0}$ the function $H_{\alpha,k}$ is polyharmonic of finite
order. If $\alpha\in\mathbb{R}$ is different from these numbers, then
$H_{\alpha,k},$ as a function on the annular region $A\left(  r_{0}%
,r_{1}\right)  $ with $r_{0}>0,$ is polyharmonic of infinite order and type at
most $1/r_{0}.$
\end{theorem}

\begin{proof}
A straightforward calculation provides the formula
\[
\Delta\left(  \left\vert x\right\vert ^{2\alpha}Y_{k}\left(  x\right)
\right)  =2\alpha\left(  2\alpha+d-2+2k\right)  \cdot\left\vert x\right\vert
^{2\alpha-2}Y_{k}\left(  x\right)  .
\]
Hence $\Delta^{p}\left(  \left\vert x\right\vert ^{2\alpha}Y_{k}\left(
x\right)  \right)  =c_{\alpha,p}\left\vert x\right\vert ^{2\alpha-2p}%
Y_{k}\left(  x\right)  $ where
\begin{align*}
c_{\alpha,p}  &  =2\alpha\left(  2\alpha-2\right)  \cdots\left(
2\alpha-2\left(  p-1\right)  \right)  \cdot\\
&  \cdot\left(  2\alpha+d-2+2k\right)  \cdots\left(  2\alpha+d-2+2k-2\left(
p-1\right)  \right)  .
\end{align*}
Thus $\Delta^{p}\left(  \left\vert x\right\vert ^{2\alpha}Y_{k}\left(
x\right)  \right)  =0$ if and only if $2\alpha-2j=0$ or $2\alpha+d-2+2k-2j=0$
for some $j=0,\ldots,p-1.$ This means that $\alpha=j$ for some $j\in\left\{
0,\ldots,p-1\right\}  $ or $\alpha=1-\frac{1}{2}d-k+j$ for some $j\in\left\{
0,\ldots,p\right\}  .$ Hence the first statement is proven.

Next consider the power series $f\left(  z\right)  =\sum_{p=1}^{\infty
}c_{\alpha,p}z^{p}/\left(  2p\right)  !$ in the complex variable $z.$ Assume
that $\alpha\neq j$ and $\alpha\neq1-\frac{1}{2}d-k+j$ for all natural numbers
$j\in\mathbb{N}_{0}.$ Then the convergence radius $R$ can be computed by the
ratio test
\[
R=\lim_{p\rightarrow\infty}\frac{c_{\alpha,p}/\left(  2p\right)  !}%
{c_{\alpha,p+1}/\left(  2p+2\right)  !}=\lim_{p\rightarrow\infty}\frac{\left(
2p+1\right)  \left(  2p+2\right)  }{\left(  2\alpha-2p\right)  \left(
2\alpha+d-2+2k-2p\right)  }=1.
\]
The convergence radius formula yields $R=\overline{\lim}_{p\rightarrow\infty
}\sqrt[p]{\left\vert c_{\alpha,p}\right\vert /\left(  2p\right)  !}=1.$
Moreover,
\begin{equation}
\sqrt[2p]{\frac{\left\vert \Delta^{p}\left(  \left\vert x\right\vert
^{2\alpha}Y_{k}\left(  x\right)  \right)  \right\vert }{\left(  2p\right)  !}%
}=\sqrt[2p]{\left\vert \frac{c_{\alpha,p}}{\left(  2p\right)  !}\right\vert
}\sqrt[2p]{\left\vert Y_{k}\left(  x\right)  \right\vert }\sqrt[2p]{\left\vert
x\right\vert ^{2\alpha}}\cdot\frac{1}{\left\vert x\right\vert }.
\label{eqpolyalpha}%
\end{equation}
Let now $K\subset A\left(  r_{0},r_{1}\right)  $ be a compact subset. Since
$Y_{k}$ is continuous it is bounded on $K,$ say by $M_{k}.$ Clearly
$\left\vert x\right\vert >r_{0}$ for $x\in K.$ Thus we can estimate
\[
\sqrt[2p]{\left\vert Y_{k}\left(  x\right)  \right\vert }\sqrt[2p]{\left\vert
x\right\vert ^{2\alpha}}\cdot\frac{1}{\left\vert x\right\vert }\leq
\sqrt[2p]{M_{k}}\sqrt[2p]{r_{1}^{2\alpha}}\cdot\frac{1}{r_{0}}%
\]
for all $x\in K$. Using this estimate in (\ref{eqpolyalpha}) and taking the
limit $p\rightarrow\infty,$ we see that $H_{\alpha,k}$ is polyharmonic of
infinite order and type at most $1/r_{0}.$
\end{proof}

\section{Linear differential operators with constant coefficients}

In this section we shall review some results about linear differential
operators with constant coefficients (see e.g.\ \cite{Godu97}). Mainly we
shall study Taylor-type expansion of a \ $C^{\infty}$-function with respect to
a linear differential operator with constant coefficients. Some material can
be found in \cite{Schu81} but we shall need a deeper analysis of this topic.
We shall give a self-contained presentation in order to facilitate the
readability of the paper and to fix notations.

Let $\lambda_{0},\ldots,\lambda_{n}$ be complex numbers, and define the linear
differential operator with constant coefficients $L$ by
\begin{equation}
L:=L_{\lambda_{0},\ldots,\lambda_{n}}:=\left(  \frac{d}{dx}-\lambda
_{0}\right)  \ldots.\left(  \frac{d}{dx}-\lambda_{n}\right)  . \label{eqdefL}%
\end{equation}
The space of all solutions of $Lu=0$ is denoted by
\begin{equation}
E_{\left(  \lambda_{0},\ldots,\lambda_{n}\right)  }:=\left\{  f\in C^{\infty
}\left(  \mathbb{R}\right)  ;Lf=0\right\}  . \label{Espace}%
\end{equation}
Elements in $E_{\left(  \lambda_{0},\ldots,\lambda_{n}\right)  }$ are called
\emph{exponential polynomials} or sometimes $L$\emph{-polynomials, }and
$\lambda_{0},\ldots,\lambda_{n}$ are called \emph{exponents} or
\emph{frequencies} (see e.g.\ Chapter 3 in \cite{BeGa95}).

In the case of \emph{pairwise different} $\lambda_{j},j=0,\ldots,n,$ the space
$E_{\left(  \lambda_{0},\ldots,\lambda_{n}\right)  }$ is the linear span
generated by the functions $e^{\lambda_{0}x},e^{\lambda_{1}x},\ldots
,e^{\lambda_{n}x}.$ In the case when $\lambda_{j}$ occurs $m_{j}$ times in
$\Lambda_{n}=\left(  \lambda_{0},\ldots,\lambda_{n}\right)  ,$ a basis of the
space $E_{\left(  \lambda_{0},\ldots,\lambda_{n}\right)  }$ is given by the
linearly independent functions
\begin{equation}
x^{s}e^{\lambda_{j}x}\qquad\text{for }s=0,1,\ldots,m_{j}-1. \label{eqexpof}%
\end{equation}
In the case that $\lambda_{0}=\cdots=\lambda_{n}$ $=0,$ the space $E_{\left(
\lambda_{0},\ldots,\lambda_{n}\right)  }$ is just the space of all polynomials
of degree at most $n,$ and we shall refer to this as the \emph{polynomial
case}.

\subsection{The fundamental function}

It is well known that for $\Lambda_{n}=\left(  \lambda_{0},\ldots,\lambda
_{n}\right)  \in\mathbb{C}^{n+1}$ there exists a unique solution
$\Phi_{\Lambda_{n}}\in E_{\left(  \lambda_{0},\ldots,\lambda_{n}\right)  }$ to
the Cauchy problem
\[
\Phi_{\Lambda_{n}}\left(  0\right)  =\ldots=\Phi_{\Lambda_{n}}^{\left(
n-1\right)  }\left(  0\right)  =0\text{ and }\Phi_{\Lambda_{n}}^{\left(
n\right)  }\left(  0\right)  =1.
\]
We shall call $\Phi_{\Lambda_{n}}$ the \emph{fundamental function} in
$E_{\left(  \lambda_{0},\ldots,\lambda_{n}\right)  }$ (see e.g.\ \cite{Micc76}%
). An explicit formula for $\Phi_{\Lambda_{n}}$ is
\begin{equation}
\Phi_{\Lambda_{n}}\left(  x\right)  :=\frac{1}{2\pi i}\int_{\Gamma_{r}}%
\frac{e^{xz}}{\left(  z-\lambda_{0}\right)  \cdots\left(  z-\lambda
_{n}\right)  }dz, \label{defPhi}%
\end{equation}
where $\Gamma_{r}$ is the path in the complex plane defined by $\Gamma
_{r}\left(  t\right)  =re^{it}$, $t\in\left[  0,2\pi\right]  $, surrounding
all the complex numbers $\lambda_{0},\ldots,\lambda_{n}$ (see Proposition
\ref{PropTaylor} below). Note that (\ref{defPhi}) implies the useful formula
\begin{equation}
\left(  \frac{d}{dx}-\lambda_{n+1}\right)  \Phi_{\left(  \lambda_{0}%
,\ldots,\lambda_{n+1}\right)  }\left(  x\right)  =\Phi_{\left(  \lambda
_{0},\ldots,\lambda_{n}\right)  }\left(  x\right)  . \label{rec0}%
\end{equation}
The fundamental function can be seen as analogue of the power function $x^{n}$
in the space $E_{\left(  \lambda_{0},\ldots,\lambda_{n}\right)  }$: In the
polynomial case, i.e., $\lambda_{0}=\ldots=\lambda_{n}=0,$ the fundamental
function is
\begin{equation}
\Phi_{\text{pol,}n}\left(  x\right)  =\frac{1}{n!}x^{n}. \label{Fundpol}%
\end{equation}
In general, explicit formulae for the fundamental function are complicated.
However, in the case of equidistant exponents one can compute the fundamental
function in a very simple way (see \cite{Li85}):

\begin{example}
\label{Exequi}Assume that $\lambda_{k}=\alpha+\omega k$ for $k=0,\ldots,n,$
and complex numbers $\omega\neq0$ and $\alpha.$ Then
\begin{equation}
\Phi_{\text{\textrm{equi,}}n}\left(  x\right)  =\frac{1}{n!}\frac{1}%
{\omega^{n}}e^{\alpha x}\left(  e^{\omega x}-1\right)  ^{n}=\frac{1}{n!}%
\frac{1}{\omega^{n}}\sum_{k=0}^{n}\binom{n}{k}e^{\left(  \alpha+k\omega
\right)  x}\left(  -1\right)  ^{n-k}.\label{eq12b}%
\end{equation}
Indeed, it is easy to see that $\Phi_{\text{\textrm{equi,}}n}\left(  0\right)
=\cdots=\Phi_{\text{\textrm{equi,}}n}^{\left(  n-1\right)  }\left(  0\right)
=0$ and $\Phi_{\text{\textrm{equi,}}n}^{\left(  n\right)  }\left(  0\right)
=1,$ and the right-hand side of (\ref{eq12b}) shows that $\Phi
_{\text{\textrm{equi,}}n}\in E_{\left(  \lambda_{0},\ldots,\lambda_{n}\right)
}$.
\end{example}

In the following we shall give estimates of the fundamental function which
seem to be new. Our estimates are based on the Taylor expansion of the
fundamental function $\Phi_{\Lambda_{n}}$ which will be described as follows:

\begin{proposition}
\label{PropTaylor}The function $\Phi_{\left(  \lambda_{0},\ldots,\lambda
_{n}\right)  }$ defined in (\ref{defPhi}) satisfies $\Phi_{\left(  \lambda
_{0},\ldots,\lambda_{n}\right)  }^{\left(  k\right)  }\left(  0\right)  =0$
for $k=0,\ldots,n-1$. For $k\geq n$ the formula
\begin{equation}
\Phi_{\left(  \lambda_{0},\ldots,\lambda_{n}\right)  }^{\left(  k\right)
}\left(  0\right)  =\sum_{\substack{\left(  s_{0},\ldots,s_{n}\right)
\in\mathbb{N}_{0}^{n+1}\text{ }\\s_{0}+\cdots+s_{n}+n=k,}}\lambda_{0}^{s_{0}%
}\cdots\lambda_{n}^{s_{n}} \label{eqTaylorcoeffphi}%
\end{equation}
holds. In particular, $\Phi_{\left(  \lambda_{0},\ldots,\lambda_{n}\right)
}^{\left(  n\right)  }\left(  0\right)  =1$ and $\Phi_{\left(  \lambda
_{0},\ldots,\lambda_{n}\right)  }^{\left(  n+1\right)  }\left(  0\right)
=\lambda_{0}+\cdots+\lambda_{n}.$
\end{proposition}

\begin{proof}
For $z\in\mathbb{C}$ with $\left\vert z\right\vert >\left\vert \lambda
_{j}\right\vert ,$ the geometric series
\[
\frac{1}{z-\lambda_{j}}=\frac{1}{z}\cdot\frac{1}{1-\lambda_{j}/z}=\sum
_{s=0}^{\infty}\lambda_{j}^{s}\left(  \frac{1}{z}\right)  ^{s+1}%
\]
converges. Thus we obtain from (\ref{defPhi}) that
\[
\Phi_{\left(  \lambda_{0},\ldots,\lambda_{n}\right)  }\left(  x\right)
=\sum_{s_{0}=0}^{\infty}\cdots\sum_{s_{n}=0}^{\infty}\frac{1}{2\pi i}%
\int_{\Gamma_{r}}\lambda_{0}^{s_{0}}\ldots\lambda_{n}^{s_{n}}\frac{e^{xz}%
}{z^{s_{0}+\cdots+s_{n}+n+1}}dz.
\]
By differentiating one obtains
\[
\Phi_{\left(  \lambda_{0},\ldots,\lambda_{n}\right)  }^{\left(  k\right)
}\left(  x\right)  =\sum_{s_{0}=0}^{\infty}\cdots\sum_{s_{n}=0}^{\infty}%
\frac{1}{2\pi i}\int_{\Gamma_{r}}\lambda_{0}^{s_{0}}\cdots\lambda_{n}^{s_{n}%
}\frac{z^{k}e^{xz}}{z^{s_{0}+\cdots+s_{n}+n+1}}dz.
\]
For $x=0$ the integral is easy to evaluate and the result is proven.
\end{proof}

In the following proposition we give the first estimate for the fundamental function:

\begin{proposition}
\label{PropPhi}Let $\lambda_{0},\ldots,\lambda_{n}$ be complex numbers and
$M_{n}:=\max\left\{  \left\vert \lambda_{j}\right\vert ;j=0,\ldots,n\right\}
.$ Then the inequality
\begin{equation}
\left\vert \Phi_{\left(  \lambda_{0},\ldots,\lambda_{n}\right)  }\left(
z\right)  \right\vert \leq\Phi_{\left(  \left\vert \lambda_{0}\right\vert
,\ldots,\left\vert \lambda_{n}\right\vert \right)  }\left(  \left\vert
z\right\vert \right)  \leq\frac{1}{n!}\left\vert z\right\vert ^{n}%
e^{M_{n}\left\vert z\right\vert }\text{ } \label{eqphiin}%
\end{equation}
holds for all $z\in\mathbb{C}.$
\end{proposition}

\begin{proof}
Using (\ref{eqTaylorcoeffphi}) we can estimate the Taylor coefficient
\[
\left\vert \Phi_{\left(  \lambda_{0},\ldots,\lambda_{n}\right)  }^{\left(
k\right)  }\left(  0\right)  \right\vert \leq\sum_{s_{0}+\cdots+s_{n}%
+n=k}\left\vert \lambda_{0}^{s_{0}}\cdots\lambda_{n}^{s_{n}}\right\vert
=\sum_{s_{0}+\cdots+s_{n}+n=k}^{\infty}\left\vert \lambda_{0}\right\vert
^{s_{0}}\cdots\left\vert \lambda_{n}\right\vert ^{s_{n}}%
\]
which is obviously equal to $\Phi_{\left(  \left\vert \lambda_{0}\right\vert
,\ldots,\left\vert \lambda_{n}\right\vert \right)  }^{\left(  k\right)
}\left(  0\right)  .$ Since $\Phi_{\Lambda_{n}}\left(  z\right)  =\sum
_{k=n}^{\infty}\Phi_{\left(  \lambda_{0},\ldots,\lambda_{n}\right)  }^{\left(
k\right)  }\left(  0\right)  z^{k}/k!$ we can estimate
\[
\left\vert \Phi_{\left(  \lambda_{0},\ldots,\lambda_{n}\right)  }\left(
z\right)  \right\vert \leq\sum_{k=n}^{\infty}\frac{1}{k!}\left\vert
\Phi_{\left(  \lambda_{0},\ldots,\lambda_{n}\right)  }^{\left(  k\right)
}\left(  0\right)  \right\vert \cdot\left\vert z\right\vert ^{k}=\Phi_{\left(
\left\vert \lambda_{0}\right\vert ,\ldots,\left\vert \lambda_{n}\right\vert
\right)  }\left(  \left\vert z\right\vert \right)  .
\]
Since $\left\vert \lambda_{j}\right\vert \leq M_{n}$ for all $j=0,\ldots,n$,
we can estimate $\left\vert \lambda_{0}\right\vert ^{s_{0}}\cdots\left\vert
\lambda_{n}\right\vert ^{s_{n}}\leq M_{n}^{s_{0}+\cdots+s_{n}}.$ Using
(\ref{eqTaylorcoeffphi}) for $\left\vert \lambda_{0}\right\vert ,\ldots
,\left\vert \lambda_{n}\right\vert $ we obtain for $k\geq n$
\[
\Phi_{\left(  \left\vert \lambda_{0}\right\vert ,\ldots,\left\vert \lambda
_{n}\right\vert \right)  }^{\left(  k\right)  }\left(  0\right)  =\sum
_{s_{0}+\cdots+s_{n}=k-n}\left\vert \lambda_{0}\right\vert ^{s_{0}}%
\cdots\left\vert \lambda_{n}\right\vert ^{s_{n}}\leq\binom{k-n+n}{k-n}%
M_{n}^{k-n}.
\]
We conclude that
\begin{align*}
\Phi_{\left(  \left\vert \lambda_{0}\right\vert ,\ldots,\left\vert \lambda
_{n}\right\vert \right)  }\left(  \left\vert z\right\vert \right)   &
\leq\sum_{k=n}^{\infty}\frac{\left\vert z\right\vert ^{k}}{k!}\binom{k}%
{k-n}M_{n}^{k-n}=\sum_{k=0}^{\infty}\frac{\left\vert z\right\vert ^{k+n}%
}{\left(  k+n\right)  !}\binom{k+n}{k}M_{n}^{k}\\
&  =\frac{1}{n!}\left\vert z\right\vert ^{n}\sum_{k=0}^{\infty}\frac{1}%
{k!}\left\vert M_{n}z\right\vert ^{k}=\frac{1}{n!}\left\vert z\right\vert
^{n}e^{M_{n}\left\vert z\right\vert }.
\end{align*}
The proof is accomplished.
\end{proof}

Suppose now that $\lambda_{0},\lambda_{1},\ldots,$ is a \emph{bounded}
sequence of complex numbers. Then (\ref{eqphiin}) implies that
\begin{equation}
\overline{\lim_{n\rightarrow\infty}}\sqrt[n]{n!\left\vert \Phi_{\Lambda_{n}%
}\left(  z\right)  \right\vert }\leq\left\vert z\right\vert .
\label{eqboundedlambda}%
\end{equation}
In the special case that the exponents $\lambda_{n}$ are equal to $0,$ a
stronger conclusion is valid. Namely, by using the explicit formula
(\ref{Fundpol}), we know that the limit in (\ref{eqboundedlambda}) exists and
\[
\lim_{n\rightarrow\infty}\sqrt[n]{n!\left\vert \Phi_{\text{pol,}n}\left(
z\right)  \right\vert }=\lim_{n\rightarrow\infty}\sqrt[n]{\left\vert
z^{n}\right\vert }=\left\vert z\right\vert .
\]
Next suppose that the estimate $\left\vert \lambda_{n}\right\vert \leq\beta n$
holds for all natural numbers and some $\beta>0.$ Then (\ref{eqphiin}) yields
the estimate
\begin{equation}
\overline{\lim_{n\rightarrow\infty}}\sqrt[n]{n!\left\vert \Phi_{\Lambda_{n}%
}\left(  z\right)  \right\vert }\leq\left\vert z\right\vert e^{\beta\left\vert
z\right\vert }. \label{eqlimsu}%
\end{equation}
The estimate (\ref{eqlimsu}) seems to be satisfactory. However, the example of
the equidistant points $\lambda_{n}=n\omega$ shows that this is not the
optimal estimate. Namely, by using (\ref{eq12b}) we infer that
\[
\lim_{n\rightarrow\infty}\sqrt[n]{n!\left\vert \Phi_{\text{\textrm{equi,}}%
n}\left(  z\right)  \right\vert }=\frac{1}{\left\vert \omega\right\vert
}\left\vert e^{\omega z}-1\right\vert \leq\left\vert z\right\vert
e^{\left\vert \omega\right\vert \left\vert z\right\vert }.
\]
Next we shall provide a similar estimate for general exponents $\lambda_{n}$
obeying an estimate of the form $\left\vert \lambda_{n}\right\vert \leq\beta
n.$

\begin{proposition}
\label{Propmonot}Let $\lambda_{j}$ and $\mu_{j}$ be real numbers satisfying
$0\leq\lambda_{j}\leq\mu_{j}$ for $j=0,\ldots,n.$ Then $\Phi_{\Lambda_{n}%
}\left(  x\right)  $ is real for all $x\in\mathbb{R}$ and $\Phi_{\Lambda_{n}%
}\left(  x\right)  >0$ for all $x>0.$ Moreover
\[
\left\vert \Phi_{\left(  \lambda_{0},\ldots,\lambda_{n}\right)  }\left(
z\right)  \right\vert \leq\Phi_{\left(  \mu_{0},\ldots,\mu_{n}\right)
}\left(  \left\vert z\right\vert \right)
\]
for all complex numbers $z.$
\end{proposition}

\begin{proof}
By (\ref{eqTaylorcoeffphi}) the Taylor coefficients of $\Phi_{\left(
\lambda_{0},\ldots,\lambda_{n}\right)  }\left(  x\right)  $ are real, so
$\Phi_{\left(  \lambda_{0},\ldots,\lambda_{n}\right)  }\left(  x\right)  $ is
a real number for real $x.$ Clearly $0\leq\lambda_{j}\leq\mu_{j}$ implies that
$0\leq\lambda_{j}^{s_{j}}\leq\mu_{j}^{s_{j}}$ for any natural number $s_{j}$,
$j=0,\ldots,n.$ Thus $0\leq\lambda_{0}^{s_{0}}\cdots\lambda_{n}^{s_{n}}\leq
\mu_{0}^{s_{0}}\cdots\mu_{n}^{s_{n}}$ for any $\left(  s_{0},\ldots
,s_{n}\right)  \in\mathbb{N}_{0}^{n+1}.$ By formula (\ref{eqTaylorcoeffphi})
we have
\[
0\leq\Phi_{\left(  \lambda_{0},\ldots,\lambda_{n}\right)  }^{\left(  k\right)
}\left(  0\right)  \leq\Phi_{\left(  \mu_{0},\ldots,\mu_{n}\right)  }^{\left(
k\right)  }\left(  0\right)
\]
for all $k\in\mathbb{N}_{0}.$ It follows that $0\leq\Phi_{\left(  \lambda
_{0},\ldots,\lambda_{n}\right)  }\left(  \left\vert z\right\vert \right)
\leq\Phi_{\left(  \mu_{0},\ldots,\mu_{n}\right)  }\left(  \left\vert
z\right\vert \right)  .$ The proof is finished by combing the last inequality
with (\ref{eqphiin}).
\end{proof}

\begin{theorem}
\label{ThmMainCon}Let $\lambda_{n},n\in\mathbb{N}_{0},$ be complex numbers
such that $\overline{\lim}_{n\rightarrow\infty}\left\vert \lambda
_{n}\right\vert /n\leq\beta.$ Then for any $\varepsilon>0$ there exists a
number $\alpha>0$ such that
\begin{equation}
n!\left\vert \Phi_{\Lambda_{n}}\left(  z\right)  \right\vert \leq
e^{\alpha\left\vert z\right\vert }\left(  \frac{e^{\left(  1+\varepsilon
\right)  \beta\left\vert z\right\vert }-1}{\left(  1+\varepsilon\right)
\beta}\right)  ^{n} \label{infanta}%
\end{equation}
for all natural numbers $n$ and for all complex numbers $z.$ In other words,
\begin{equation}
\overline{\lim_{n\rightarrow\infty}}\sqrt[n]{n!\left\vert \Phi_{\Lambda_{n}%
}\left(  z\right)  \right\vert }\leq\frac{e^{\beta\left\vert z\right\vert }%
-1}{\beta}. \label{infanta2}%
\end{equation}

\end{theorem}

\begin{proof}
Let $\varepsilon>0.$ Then there exist $n_{0}$ such that $\left\vert
\lambda_{n}\right\vert \leq\left(  1+\varepsilon\right)  \beta n$ for all
$n\geq n_{0}.$ Take $\alpha>0$ large enough so that $\left\vert \lambda
_{n}\right\vert \leq$ $\left(  1+\varepsilon\right)  \beta n+\alpha$ for all
natural numbers $n.$ Define $\mu_{n}:=\alpha+\left(  1+\varepsilon\right)
\beta n$ for all $n,$ so $\left\vert \lambda_{n}\right\vert \leq\mu_{n}$ for
all $n\in\mathbb{N}_{0}.$ Propositions \ref{PropPhi} and \ref{Propmonot} and
Example \ref{Exequi} show that
\[
\left\vert \Phi_{\left(  \lambda_{0},\ldots,\lambda_{n}\right)  }\left(
z\right)  \right\vert \leq\Phi_{\left(  \left\vert \lambda_{0}\right\vert
,\ldots,\left\vert \lambda_{n}\right\vert \right)  }\left(  \left\vert
z\right\vert \right)  \leq\Phi_{\left(  \mu_{0},\ldots,\mu_{n}\right)
}\left(  \left\vert z\right\vert \right)  =\frac{1}{n!}e^{\alpha\left\vert
z\right\vert }\left(  \frac{e^{\left(  1+\varepsilon\right)  \beta\left\vert
z\right\vert }-1}{\left(  1+\varepsilon\right)  \beta}\right)  ^{n}.
\]
This shows (\ref{infanta}) and clearly (\ref{infanta2}) is a simple
consequence of (\ref{infanta}).
\end{proof}

The next theorem is our main result in this subsection and it will be used in
later sections:

\begin{theorem}
\label{ThmPower}Let $\beta>0$ and $\lambda_{n},n\in\mathbb{N}_{0},$ be complex
numbers such that $\overline{\lim}_{n\rightarrow\infty}\left\vert \lambda
_{n}\right\vert /n\leq\beta.$ Let $a_{n}$ be complex numbers for
$n\in\mathbb{N}_{0}$ and define $R^{\ast}$ through
\begin{equation}
\frac{1}{R^{\ast}}=\overline{\lim_{n\rightarrow\infty}}\sqrt[n]{\left\vert
\frac{a_{n}}{n!}\right\vert }. \label{eqconvradius}%
\end{equation}
If $R^{\ast}>0,$ then the series $\sum_{n=0}^{\infty}a_{n}\Phi_{\Lambda_{n}%
}\left(  z-x_{0}\right)  $ converges compactly and absolutely in the ball with
center $x_{0}$ and radius
\[
\frac{1}{\beta}\ln\left(  1+\beta R^{\ast}\right)  .
\]

\end{theorem}

\begin{proof}
Let $\rho<\left(  1/\beta\right)  \ln\left(  1+\beta R^{\ast}\right)  .$ Then
$\left(  e^{\beta\rho}-1\right)  /R^{\ast}<\beta.$ Take now $\varepsilon>0$
small enough so that
\begin{equation}
\frac{e^{\left(  1+\varepsilon\right)  \beta\rho}-1}{\left(  1+\varepsilon
\right)  \left(  R^{\ast}-\varepsilon\right)  }<\beta. \label{eqest1000}%
\end{equation}
Since $\frac{1}{R^{\ast}}<\frac{1}{R^{\ast}-\varepsilon},$ formula
(\ref{eqconvradius}) shows that there exists a natural number $n_{0}$ such
that
\[
\left\vert \frac{a_{n}}{n!}\right\vert \leq\left(  \frac{1}{R^{\ast
}-\varepsilon}\right)  ^{n}%
\]
for all $n\geq n_{0}.$ By Theorem \ref{ThmMainCon} there exists a natural
number $\alpha>0$ such that
\[
n!\left\vert \Phi_{\Lambda_{n}}\left(  z\right)  \right\vert \leq
e^{\alpha\left\vert z\right\vert }\left(  \frac{e^{\left(  1+\varepsilon
\right)  \beta\left\vert z\right\vert }-1}{\left(  1+\varepsilon\right)
\beta}\right)  ^{n}%
\]
for all complex numbers $z$ and for all natural numbers $n.$ The last two
inequalities lead to
\[
\sum_{n=n_{0}}^{\infty}\left\vert a_{n}\Phi_{\Lambda_{n}}\left(
z-x_{0}\right)  \right\vert \leq e^{\alpha\rho}\sum_{n=n_{0}}^{\infty}\left(
\frac{e^{\left(  1+\varepsilon\right)  \beta\rho}-1}{\left(  1+\varepsilon
\right)  \beta\left(  R^{\ast}-\varepsilon\right)  }\right)  ^{n}%
\]
valid for all $z$ with $\left\vert z-x_{0}\right\vert \leq\rho.$ This series
converges in view of the estimate (\ref{eqest1000}).
\end{proof}

\textbf{Example:} Let $\lambda_{n}=n+1$ for $n\in\mathbb{N}_{0},$ and consider
the constant function $f\left(  x\right)  =1.$ Then
\[
a_{n}:=\left(  \frac{d}{dx}-\lambda_{0}\right)  \cdots\left(  \frac{d}%
{dx}-\lambda_{n-1}\right)  f\left(  x_{0}\right)  =\left(  -1\right)
^{n}\lambda_{0}\cdots\lambda_{n-1}=\left(  -1\right)  ^{n}n!.
\]
Thus $\lim_{n\rightarrow\infty}\sqrt[n]{\left\vert a_{n}\right\vert /n!}=1.$
Further $\Phi_{\Lambda_{n}}\left(  x\right)  =e^{x}\left(  e^{x}-1\right)
^{n}/n!.$ According to Theorem \ref{ThmPower} (with $x_{0}=0$ and $\beta=1)$
the series
\begin{equation}
\sum_{n=0}^{\infty}a_{n}\Phi_{\Lambda_{n}}\left(  x\right)  =\sum
_{n=0}^{\infty}\left(  -1\right)  ^{n}e^{x}\left(  e^{x}-1\right)  ^{n}%
=e^{x}\frac{1}{1-\left(  1-e^{x}\right)  }=1 \label{eqextaylor}%
\end{equation}
converges for all complex numbers $z$ with $\left\vert z\right\vert <\ln2.$ Of
course, this can be seen directly for real $x.$ Namely, if $e^{x}-1<1$ (which
means that $e^{x}<2,$ so $x<\ln2),$ the series obviously converges. On the
other hand, for $e^{x}-1\geq1$ we do not have convergence. Theorem
\ref{ThmTaylor2} below provides the following interpretation: The constant
function $1$ has the Taylor series expansion (\ref{eqextaylor}) on $\left(
-\ln2,\ln2\right)  $ with respect to the differential operators $\left(
d/dx-\lambda_{0}\right)  \cdots\left(  d/dx-\lambda_{n-1}\right)  .$

An analogue of Theorem \ref{ThmPower} can be proved for a \emph{bounded}
sequence of exponents $\lambda_{n}$ (see Theorem \ref{ThmPower2} below),
either by using Proposition \ref{PropPhi} or by applying Theorem
\ref{ThmPower} for $\beta>0$ arbitrary. In the latter case, the rule of
L'Hospital
\[
\lim_{\beta\rightarrow0}\frac{\ln\left(  1+\beta R^{\ast}\right)  }{\beta
}=\lim_{\beta\rightarrow0}\frac{R^{\ast}}{1+\beta R^{\ast}}=R^{\ast}%
\]
can be used for computing the correct radius of convergence:

\begin{theorem}
\label{ThmPower2} Let $\lambda_{n},n\in\mathbb{N}_{0},$ be a bounded sequence
of complex numbers. Let $a_{n}$ be complex numbers for $n\in\mathbb{N}_{0}$
and define $R^{\ast}$ as in (\ref{eqconvradius}). If $R^{\ast}>0$ then the
series $\sum_{n=0}^{\infty}a_{n}\Phi_{\Lambda_{n}}\left(  z-x_{0}\right)  $
converges compactly and absolutely in the ball with center $x_{0}$ and radius
$R^{\ast}.$
\end{theorem}

It is a natural and interesting question whether in (\ref{infanta2}) the
\emph{limit} exists. We mention two results addressing this problem but we
omit the proofs since we shall not need them in the following.

\begin{theorem}
Suppose that $\lambda_{n},n\in\mathbb{N}_{0},$ is a bounded sequence of real
numbers and $\Lambda_{n}=\left(  \lambda_{0},\ldots,\lambda_{n}\right)  .$
Then the following limit exists for all $x\geq0$:
\[
\lim_{n\rightarrow\infty}\sqrt[n]{n!\Phi_{\Lambda_{n}}\left(  x\right)  }=x.
\]

\end{theorem}

\begin{theorem}
Let $\lambda_{n},n\in\mathbb{N}_{0}$ be a sequence of real numbers such that
the limit $\lim_{n\rightarrow\infty}\lambda_{n}/n=\beta$ exists. Then the
following limit exists for all $x\geq0$:
\[
\lim_{n\rightarrow\infty}\sqrt[n]{n!\Phi_{\Lambda_{n}}\left(  x\right)
}=\frac{e^{\beta x}-1}{\beta}.
\]

\end{theorem}

\subsection{Taylor series for linear differential operators with constant
coefficients}

Let $\lambda_{0},\ldots\lambda_{n}$ be complex numbers. As analogue of the
$n$-th derivative $f^{\left(  n\right)  }\left(  t\right)  $ of a function
$f\left(  t\right)  ,$ we define in the setting of linear differential
operators
\[
D^{\left(  n\right)  }f\left(  t\right)  :=\left(  \frac{d}{dt}-\lambda
_{0}\right)  \cdots\left(  \frac{d}{dt}-\lambda_{n-1}\right)  f\left(
t\right)  .
\]
To avoid overburdened indexes, we dropped the dependence on $\lambda_{j}$ in
the above notation $D^{\left(  n\right)  }f\left(  t\right)  .$ For $n=0$ we
define $D^{\left(  0\right)  }f\left(  t\right)  =f\left(  t\right)  .$ We
shall also use the notation
\[
D_{\lambda}f\left(  t\right)  :=\frac{d}{dt}f\left(  t\right)  -\lambda
f\left(  t\right)
\]
which should be distinguished from the notation $D^{\left(  n\right)  }f.$

The main result of this subsection is Theorem \ref{ThmTaylor2} providing a
Taylor type expansion of a smooth function according to the fundamental
functions $\Phi_{n}\left(  x-x_{0}\right)  $. As a preparation we need the
following well-known result whose proof is included for convenience of the
reader (cf. e.g.\ \cite{Schu81}).

\begin{theorem}
Let $\lambda_{0},\lambda_{1},\ldots,\lambda_{n}$ be complex numbers and define
$\Lambda_{k}=\left(  \lambda_{0},\ldots,\lambda_{k}\right)  $ for
$k\in\left\{  0,\ldots,n\right\}  .$ Assume that $f:\left[  x_{0},x_{0}%
+\gamma\right]  \rightarrow\mathbb{C}$ is $C^{n+1}$ for some $\gamma>0$. Then
for any $m\leq n$ and $x\in\left[  x_{0},x_{0}+\gamma\right]  $
\begin{equation}
f\left(  x\right)  =\sum_{k=0}^{m}D^{\left(  k\right)  }f\left(  x_{0}\right)
\Phi_{\Lambda_{k}}\left(  x-x_{0}\right)  +\int_{x_{0}}^{x}D^{\left(
m+1\right)  }f\left(  t\right)  \cdot\Phi_{\Lambda_{m}}\left(  x-t\right)  dt.
\label{eqtaylor}%
\end{equation}

\end{theorem}

\begin{proof}
We shall prove the statement by induction over $m.$ For $m=0$ this means that
\[
f\left(  x\right)  =f\left(  x_{0}\right)  \Phi_{(\lambda_{0})}\left(
x-x_{0}\right)  +\int_{x_{0}}^{x}\left(  \frac{d}{dt}-\lambda_{0}\right)
f\left(  t\right)  \cdot\Phi_{(\lambda_{0})}\left(  x-t\right)  dt.
\]
Since $\Phi_{(\lambda_{0})}\left(  x\right)  =e^{\lambda_{0}x},$ this is
equivalent to
\[
f\left(  x\right)  -e^{\lambda_{0}\left(  x-x_{0}\right)  }f\left(
x_{0}\right)  =\int_{x_{0}}^{x}\left(  \frac{d}{dt}-\lambda_{0}\right)
f\left(  t\right)  \cdot e^{\lambda_{0}\left(  x-t\right)  }dt=e^{\lambda
_{0}x}\int_{x_{0}}^{x}\frac{d}{dt}\left(  e^{-\lambda_{0}t}f\left(  t\right)
\right)  dt,
\]
which is obviously true since
\[
\frac{d}{dt}\left(  e^{-\lambda_{0}t}f\left(  t\right)  \right)
=e^{-\lambda_{0}t}\left(  \frac{d}{dt}-\lambda_{0}\right)  f\left(  t\right)
.
\]
Suppose now that the statement is true for $m<n$ and we want to prove it for
$m+1\leq n.$ It suffices to prove that
\[
A_{m}:=\int_{x_{0}}^{x}D^{\left(  m+1\right)  }f\left(  t\right)  \cdot
\Phi_{\Lambda_{m}}\left(  x-t\right)  dt
\]
is equal to
\[
B_{m}:=D^{\left(  m+1\right)  }f\left(  x_{0}\right)  \Phi_{\Lambda_{m+1}%
}\left(  x-x_{0}\right)  +\int_{x_{0}}^{x}D^{\left(  m+2\right)  }f\left(
t\right)  \cdot\Phi_{\Lambda_{m+1}}\left(  x-t\right)  dt.
\]
Using the recursion $\Phi_{\Lambda_{m+1}}^{\prime}\left(  t\right)
=\Phi_{\Lambda_{m}}\left(  t\right)  +\lambda_{m+1}\Phi_{\Lambda_{m+1}}\left(
t\right)  $ in (\ref{rec0}), one obtains
\[
\frac{d}{dt}\left(  \Phi_{\Lambda_{m+1}}\left(  x-t\right)  \right)
=-\Phi_{\Lambda_{m+1}}^{\prime}\left(  x-t\right)  =-\Phi_{\Lambda_{m}}\left(
x-t\right)  -\lambda_{m+1}\Phi_{\Lambda_{m+1}}\left(  x-t\right)
\]
and therefore $D_{-\lambda_{m+1}}\Phi_{\Lambda_{m+1}}\left(  x-t\right)
=-\Phi_{\Lambda_{m}}\left(  x-t\right)  .$ Thus
\[
A_{m}=-\int_{x_{0}}^{x}D^{\left(  m+1\right)  }f\left(  t\right)  \cdot
D_{-\lambda_{m+1}}\left(  \Phi_{\Lambda_{m+1}}\left(  x-t\right)  \right)
dt.
\]
Proposition \ref{PropPI} below applied to $g=\Phi_{\Lambda_{m+1}}\left(
x-t\right)  ,$ $f=D^{\left(  m+1\right)  }f\left(  t\right)  $ and
$\lambda=\lambda_{m+1}$ gives
\[
A_{m}=-D^{\left(  m+1\right)  }f\left(  t\right)  \Phi_{\Lambda_{m+1}}\left(
x-t\right)  \mid_{x_{0}}^{x}+\int_{x_{0}}^{x}D_{\lambda_{m+1}}D^{\left(
m+1\right)  }f\left(  t\right)  \cdot\Phi_{\Lambda_{m+1}}\left(  x-t\right)
dt,
\]
and the result is proven since $D_{\lambda_{m+1}}D^{\left(  m+1\right)
}f\left(  t\right)  =D^{\left(  m+2\right)  }f\left(  t\right)  .$
\end{proof}

\begin{proposition}
\label{PropPI}Let $\lambda$ be a complex number and let $f,g:\left[
a,b\right]  \rightarrow\mathbb{C}$ be continuously differentiable. Then for
any $x_{0},x\in\left[  a,b\right]  $ with $x>x_{0},$ holds
\[
\int_{x_{0}}^{x}f\left(  t\right)  \cdot D_{-\lambda}g\left(  t\right)
\ dt=f\left(  t\right)  \cdot g\left(  t\right)  \mid_{x_{0}}^{x}-\int_{x_{0}%
}^{x}D_{\lambda}f\left(  t\right)  \cdot g\left(  t\right)  dt.
\]

\end{proposition}

\begin{proof}
Partial integration yields $\int_{x_{0}}^{x}fg^{\prime}dt=f\cdot g\mid_{x_{0}%
}^{x}-\int_{x_{0}}^{x}f^{\prime}gdt.$ Then
\[
\int_{x_{0}}^{x}f\left(  t\right)  \cdot D_{-\lambda}g\left(  t\right)
\ dt=\int_{x_{0}}^{x}f\left(  t\right)  \cdot\ \left(  g^{\prime}\left(
t\right)  +\lambda g\right)  dt=f\cdot g\mid_{x_{0}}^{x}-\int_{x_{0}}%
^{x}f^{\prime}gdt+\lambda\int_{x_{0}}^{x}f\left(  t\right)  \cdot\ gdt,
\]
which gives the statement.
\end{proof}

The next result gives a simple sufficient condition such that the "Taylor
polynomial", defined by (\ref{eqkey2}) below, converges to $f.$ This criterion
is based on estimates of derivatives $D^{\left(  2n\right)  }f\left(
t\right)  $ of even order motivated by the results in Section 2 for the
Fourier-Laplace coefficients of a polyharmonic function of infinite order. It
is also instructive to compare the result with the classical polynomial case
(see e.g.\ \cite{KhSh94} for a different approach).

\begin{theorem}
\label{ThmTaylor1}Let $\lambda_{n},n\in\mathbb{N}_{0},$ be complex numbers
such that $\overline{\lim}_{n\rightarrow\infty}\left|  \lambda_{n}\right|
/n\leq\beta$ for some $\beta\geq0$. Assume that $f\in C^{\infty}\left[
x_{0},x_{0}+\gamma\right]  $ with $\gamma>0$ satisfies the following property:
there exist constants $\sigma\geq0$ and $C>0$ such that
\begin{equation}
\left|  D^{\left(  2n\right)  }f\left(  t\right)  \right|  \leq C\cdot\left(
2n\right)  !\cdot\sigma^{2n} \label{eqkey}%
\end{equation}
for all $t\in\left[  x_{0},x_{0}+\gamma\right]  $ and $n\in\mathbb{N}_{0}.$
Then
\begin{equation}
s_{2n-1}\left(  x\right)  :=\sum_{k=0}^{2n-1}D^{\left(  k\right)  }f\left(
x_{0}\right)  \Phi_{\Lambda_{k}}\left(  x-x_{0}\right)  \label{eqkey2}%
\end{equation}
converges uniformly to $f\left(  x\right)  $ on the interval $\left[
x_{0},x_{0}+\delta\right]  $ for a suitable positive $\delta<\gamma.$
\end{theorem}

\begin{proof}
Define $s_{n}=\sum_{k=0}^{n}D^{\left(  k\right)  }f\left(  x_{0}\right)
\Phi_{\Lambda_{k}}\left(  x-x_{0}\right)  .$ Then $f\left(  x\right)
=s_{n}\left(  x\right)  +R_{n}\left(  x\right)  $ by Taylor's formula
(\ref{eqtaylor}) where
\begin{equation}
R_{n}\left(  x\right)  =\int_{x_{0}}^{x}D^{\left(  n+1\right)  }f\left(
t\right)  \cdot\Phi_{\Lambda_{n}}\left(  x-t\right)  dt. \label{eqrm}%
\end{equation}
For the convergence of $s_{2n-1}\left(  x\right)  $ to $f\left(  x\right)  ,$
it suffices to show that $R_{2n-1}\left(  x\right)  \rightarrow0$ for
$n\rightarrow\infty.$ Note that the integration parameter $t$ in (\ref{eqrm})
satisfies $x_{0}\leq t\leq x,$ so we have $x-t\geq0$ and $0\leq x-t\leq
x-x_{0}.$ We shall show uniform convergence $R_{2n-1}\left(  x\right)
\rightarrow0$ for all $x\in\left[  x_{0},x_{0}+\delta\right]  $ where
$\delta>0$ will be specified later. Clearly we have $\left\vert x-t\right\vert
\leq x-x_{0}\leq x_{0}+\delta-x_{0}\leq\delta.$ By Theorem \ref{ThmMainCon},
for given $\varepsilon>0$ there exists a natural number $\alpha>0$ such that
\[
\left\vert \Phi_{\Lambda_{2n-1}}\left(  x-t\right)  \right\vert \leq
\frac{e^{\alpha\left\vert x-t\right\vert }}{\left(  2n-1\right)  !}\left(
\frac{e^{\left(  1+\varepsilon\right)  \beta\left\vert x-t\right\vert }%
-1}{\left(  1+\varepsilon\right)  \beta}\right)  ^{2n-1}\leq\frac
{e^{\delta\alpha}}{\left(  2n-1\right)  !}\left(  \frac{e^{\left(
1+\varepsilon\right)  \beta\delta}-1}{\left(  1+\varepsilon\right)  \beta
}\right)  ^{2n-1}%
\]
for all natural numbers $n.$ This in connection with (\ref{eqkey}) leads to
the estimate:
\[
\left\vert R_{2n-1}\left(  x\right)  \right\vert \leq C\left\vert
x-x_{0}\right\vert \left(  2n\right)  !\sigma^{2n}\frac{e^{\delta\alpha}%
}{\left(  2n-1\right)  !}\left(  \frac{e^{\left(  1+\varepsilon\right)
\beta\delta}-1}{\left(  1+\varepsilon\right)  \beta}\right)  ^{2n-1}.
\]
Now we make $\delta>0$ so small such that $\sigma\left(  e^{\left(
1+\varepsilon\right)  \beta\delta}-1\right)  /\left(  1+\varepsilon\right)
\beta<1.$ Then $R_{2n-1}\left(  x\right)  $ converges uniformly on $\left[
x_{0},x_{0}+\delta\right]  $ to zero.
\end{proof}

\begin{proposition}
\label{Propevenodd}Let $\lambda_{n},n\in\mathbb{N}_{0},$ be real numbers such
that $\overline{\lim}_{n\rightarrow\infty}\left\vert \lambda_{n}\right\vert
/n\leq\beta$ for some $\beta>0$. Let $f\in C^{\infty}\left[  x_{0}%
,x_{0}+\gamma\right]  $ with $\gamma>0$ and assume that there exist constants
$C>0$ and $\sigma>0$ such that
\begin{equation}
\left\vert D^{\left(  2n\right)  }f\left(  t\right)  \right\vert \leq
C\cdot\left(  2n\right)  !\sigma^{2n}\text{ for all }t\in\left[  x_{0}%
,x_{0}+\gamma\right]  . \label{eqesteven}%
\end{equation}
Then for every $\varepsilon>0$ there exist constants $C_{2}>0$ and $\delta>0$
such that
\begin{equation}
\left\vert D^{\left(  2n+1\right)  }f\left(  t\right)  \right\vert \leq
C_{2}\left(  2n+1\right)  !\left(  \sigma+\varepsilon\right)  ^{2n+1}
\label{eqestodd}%
\end{equation}
for all $t\in\left[  x_{0},x_{0}+\delta\right]  $ and for all natural numbers
$n.$
\end{proposition}

\begin{proof}
Let $\varepsilon_{0}>0$. Then there exists $\alpha>0$ such that
\begin{equation}
\left\vert \lambda_{n}\right\vert \leq\alpha+\beta\left(  1+\varepsilon
_{0}\right)  n\text{ for all }n\in\mathbb{N}_{0}. \label{eqlambadaineq}%
\end{equation}
Let $\gamma>0$ and $\varepsilon>0$ as in the proposition. Clearly we can find
$\delta>0$ small enough so that $2\delta<\gamma,$ and
\begin{equation}
e^{2\beta\left(  1+\varepsilon_{0}\right)  \delta}\sigma<\sigma+\varepsilon.
\label{eq31}%
\end{equation}
The assumption (\ref{eqesteven}) implies the estimate
\begin{equation}
\left\vert D^{\left(  2n\right)  }f\left(  t\right)  \right\vert +\left\vert
D^{\left(  2n+2\right)  }f\left(  s\right)  \right\vert \leq C\left(
2n+2\right)  !\sigma^{2n}\left(  1+\sigma^{2}\right)  \label{eq36}%
\end{equation}
for all $s,t\in\left[  x_{0},x_{0}+2\delta\right]  .$ Theorem \ref{ThmApp} in
the appendix provides the estimate
\begin{align*}
\left\vert D^{\left(  2n+1\right)  }f\left(  x\right)  \right\vert  &
\leq2\max\left\{  \frac{2}{\delta},\delta\right\}  e^{\left(  \left\vert
\lambda_{2n}\right\vert +\left\vert \lambda_{2n+1}\right\vert \right)  \delta
}\\
&  \times\left(  \max_{t\in\left[  x_{0},x_{0}+2\delta\right]  }\left\vert
D^{\left(  2n\right)  }f\left(  t\right)  \right\vert +\max_{t\in\left[
x_{0},x_{0}+2\delta\right]  }\left\vert D^{\left(  2n+2\right)  }f\left(
t\right)  \right\vert \right)
\end{align*}
for all $x\in\left[  x_{0},x_{0}+\delta\right]  .$ Now (\ref{eqlambadaineq})
and (\ref{eq36}) imply that
\[
\left\vert D^{\left(  2n+1\right)  }f\left(  x\right)  \right\vert \leq
2\max\left\{  \frac{2}{\delta},\delta\right\}  Ce^{2\alpha\delta+\beta\left(
1+\varepsilon_{0}\right)  \delta+4\beta\left(  1+\varepsilon_{0}\right)
n\delta}\left(  2n+2\right)  !\sigma^{2n}\left(  1+\sigma^{2}\right)  \
\]
for all $x\in\left[  x_{0},x_{0}+\delta\right]  $ and for all $n\in
\mathbb{N}_{0}.$ The statement is now obvious since (\ref{eq31}) implies that
\[
\left(  2n+2\right)  e^{4\beta\left(  1+\varepsilon_{0}\right)  n\delta}%
\sigma^{2n}\leq A\left(  \sigma+\varepsilon\right)  ^{2n}%
\]
for a suitable constant $A$ and for all $n\in\mathbb{N}_{0}.$
\end{proof}

The next theorem is the main result of this subsection:

\begin{theorem}
\label{ThmTaylor2}Let $\lambda_{n},n\in\mathbb{N}_{0},$ be real numbers with
the property that $\overline{\lim}_{n\rightarrow\infty}\left\vert \lambda
_{n}\right\vert /n\leq\beta$ for some $\beta>0$. Let $f\in C^{\infty}\left[
x_{0},x_{0}+\gamma\right]  $ with $\gamma>0$ and assume that there exist
constants $C>0$ and $\sigma>0$ such that
\[
\left\vert D^{\left(  2n\right)  }f\left(  t\right)  \right\vert \leq
C\cdot\left(  2n\right)  !\sigma^{2n}%
\]
for all $t\in\left[  x_{0},x_{0}+\gamma\right]  $ and $n\in\mathbb{N}_{0}.$
Then the series
\[
\sum_{n=0}^{\infty}D^{\left(  n\right)  }f\left(  x_{0}\right)  \Phi
_{\Lambda_{n}}\left(  z-x_{0}\right)
\]
defines an analytic extension of $f$ and it converges compactly and absolutely
in the dics in $\mathbb{C}$ with center $x_{0}$ and radius
\[
\frac{1}{\beta}\ln\left(  1+\frac{\beta}{\sigma}\right)  .
\]

\end{theorem}

\begin{proof}
By Proposition \ref{Propevenodd}, for each $\varepsilon>0$ the estimate
$\left\vert D^{\left(  n\right)  }f\left(  x_{0}\right)  \right\vert \leq
C_{2}n!\left(  \sigma+\varepsilon\right)  ^{n}$ holds for all natural numbers
$n$. Thus
\[
\overline{\lim_{n\rightarrow\infty}}\sqrt[n]{\frac{\left\vert D^{\left(
n\right)  }f\left(  x_{0}\right)  \right\vert }{n!}}\leq\sigma+\varepsilon.
\]
Now let $\varepsilon$ go to $0.$ By Theorem \ref{ThmPower}, $\sum
_{n=0}^{\infty}D^{\left(  n\right)  }f\left(  x_{0}\right)  \Phi_{\Lambda_{n}%
}\left(  z-x_{0}\right)  $ converges for all $z$ as stated in the theorem. By
Theorem \ref{ThmTaylor1}, the series represents the function $f.$
\end{proof}

\section{Analytic extensions of Fourier-Laplace coefficients}

Let us recall that $f_{k,l}\left(  r\right)  $ is the Fourier-Laplace
coefficient of the function $f\in C\left(  A\left(  r_{0},r_{1}\right)
\right)  $ defined for all values $r\in\left(  r_{0},r_{1}\right)  ,$ cf.
formula (\ref{LFK}). Using the transformation $r=e^{v}$ with $v\in\left(  \log
r_{0},\log r_{1}\right)  $ we can define a function
\[
\widetilde{f}_{k,l}\left(  v\right)  :=f_{k,l}\left(  e^{v}\right)  .
\]
Let us look at a simple example:

\begin{example}
Let $f\left(  x\right)  =\log\left\vert x\right\vert $ be defined on the
annular region $\mathbb{R}^{d}\setminus\left\{  0\right\}  .$ Recalling that
$Y_{0,1}\left(  \theta\right)  =1/\sqrt{\omega_{d-1}},$ the Fourier-Laplace
coefficient $f_{0,1}$ defined in (\ref{LFK}) satisfies $f_{0,1}\left(
r\right)  =\sqrt{\omega_{d-1}}\log r$. Thus $f_{0,1}$ has an analytic
extension to the cutted complex plane $\mathbb{C}\setminus\left(
-\infty,0\right]  $ and $\widetilde{f}_{0,1}\left(  v\right)  =\sqrt
{\omega_{d-1}}\log e^{v}=\sqrt{\omega_{d-1}}v$ is defined for every complex
number $v\in\mathbb{C}.$
\end{example}

The next observation is very useful: the differential operator $L_{k}^{p}$
defined in (\ref{deflk}) in Section 2 can be transformed to a linear
differential operator with \emph{constant coefficients} in the variable $v$
for $r=e^{v}.$ We cite the following theorem (\cite[Theorem $10.34$]{Koun00}):

\begin{theorem}
\label{TMk}Let $0\leq r_{0}<r_{1}\leq\infty$ and let $g:\left(  r_{0}%
,r_{1}\right)  \rightarrow\mathbb{C}$ be a $C^{\infty}$-function. Define
\[
\widetilde{g}:\left(  \log r_{0},\log r_{1}\right)  \rightarrow\mathbb{C}%
,\qquad\text{ }\widetilde{g}\left(  v\right)  :=g\left(  e^{v}\right)  .
\]
Then
\[
\left[  L_{k}^{p}\left(  g\right)  \right]  \left(  e^{v}\right)
=e^{-2pv}\left[  M_{k,p}\left(  \widetilde{g}\right)  \right]  \left(
v\right)
\]
for any $v\in\left(  \log r_{0},\log r_{1}\right)  ,$ where
\begin{equation}
M_{k,p}=\prod\limits_{j=0}^{p-1}\left(  \frac{d}{dv}-\left(  k+2j\right)
\right)  \prod\limits_{j=0}^{p-1}\left(  \frac{d}{dv}-\left(
-k-d+2+2j\right)  \right)  . \label{Mk}%
\end{equation}

\end{theorem}

For given $k\in\mathbb{N}_{0}$ and dimension $d,$ let us define the exponents
\begin{equation}
\lambda_{2j}\left(  k,d\right)  =k+2j\text{ and }\lambda_{2j+1}\left(
k,d\right)  =-k-d+2+2j\text{ for }j\in\mathbb{N}_{0}. \label{eqlambda}%
\end{equation}
For notational simplicity we will often suppress the dependence on $k$ and $d$
and we simply write $\lambda_{n}$ with $n\in\mathbb{N}_{0}.$ In accordance
with the notations in Section $3,$ we shall define
\begin{align*}
\Phi_{n}\left(  v\right)   &  :=\Phi_{\left(  \lambda_{0},\ldots,\lambda
_{n}\right)  }\left(  v\right)  ,\\
D^{\left(  n\right)  }g\left(  v\right)   &  :=\left(  \frac{d}{dv}%
-\lambda_{0}\right)  \cdots\left(  \frac{d}{dv}-\lambda_{n-1}\right)  g\left(
v\right)  .
\end{align*}
Now we will prove the following result:

\begin{theorem}
\label{ThmGutest}Let $f:A\left(  r_{0},r_{1}\right)  \rightarrow\mathbb{C}$ be
polyharmonic of infinite order and type $\tau\geq0$ and define $\widetilde
{f}_{k,l}\left(  v\right)  :=f_{k,l}\left(  e^{v}\right)  $ for $v\in\left(
\log r_{0},\log r_{1}\right)  $. Then, given $v_{0}\in$ $\left(  \log
r_{0},\log r_{1}\right)  $ and $\varepsilon>0,$ there exists a constant $C>0$
such that
\begin{equation}
\left\vert D^{\left(  n\right)  }\widetilde{f}_{k,l}\left(  v_{0}\right)
\right\vert \leq C\cdot n!\cdot\left[  e^{v_{0}}\left(  \tau+\varepsilon
\right)  \right]  ^{n} \label{estest}%
\end{equation}
for all $n\in\mathbb{N}_{0}$ and for all $k\in\mathbb{N}_{0},l=1,\ldots
,a_{k}.$
\end{theorem}

\begin{proof}
Clearly, $\lambda_{n},n\in\mathbb{N}_{0},$ defined as above, are real numbers
with the property that $\overline{\lim}_{n\rightarrow\infty}\left\vert
\lambda_{n}\right\vert /n=1.$ Let $v\in\left(  \log r_{0},\log r_{1}\right)  $
and $\varepsilon>0.$ We want to apply Proposition \ref{Propevenodd} for
$\beta=1$ and to the function $\widetilde{f}_{k,l}$ and $\sigma=e^{v_{0}%
}\left(  \tau+\varepsilon\right)  $. Thus we want to show that the following
estimate holds: There exists $\gamma>0$ and $C>0$ (independent of $k,l)$ such
that
\[
\left\vert D^{\left(  2p\right)  }\widetilde{f}_{k,l}\left(  v\right)
\right\vert \leq C\cdot\left(  2p\right)  !\left[  e^{v_{0}}\left(
\tau+\varepsilon\right)  \right]  ^{2p}%
\]
for all $v\in\left[  v_{0},v_{0}+\gamma\right]  $ and $p\in\mathbb{N}_{0}.$
Theorem \ref{TMk} shows that
\begin{equation}
D^{\left(  2p\right)  }\widetilde{f}_{k,l}\left(  v\right)  =M_{k,p}\left(
\widetilde{f}_{k,l}\right)  \left(  v\right)  =e^{2pv}L_{k}^{p}\left(
f_{k,l}\right)  \left(  e^{v}\right)  \label{eqstar}%
\end{equation}
for all $v\in\left(  \log r_{0},\log r_{1}\right)  .$ Let us take $\gamma>0$
and $\varepsilon_{1}>0$ small enough so that $e^{v_{0}+\gamma}\left(
\tau+\varepsilon_{1}\right)  <e^{v_{0}}\left(  \tau+\varepsilon\right)  $ and
$v_{0}+\gamma<\log r_{1}.$ Theorem \ref{ThmLkpest} shows that, for
$\varepsilon_{1}>0$ and for the subinterval $\left[  v_{0},v_{0}%
+\gamma\right]  \subset\left(  \log r_{0},\log r_{1}\right)  ,$ there is
positive number $C$ such that
\[
\left\vert L_{k}^{p}\left(  f_{k,l}\right)  \left(  e^{v}\right)  \right\vert
\leq C\left(  2p\right)  !\left(  \tau+\varepsilon_{1}\right)  ^{2p}\text{ }%
\]
for all $v\in\left[  v_{0},v_{0}+\gamma\right]  $ and all $p\in\mathbb{N}%
_{0},k\in\mathbb{N}_{0},l=1,\ldots,a_{k}.$ Since $e^{v}\leq e^{v_{0}+\gamma},$
we obtain from (\ref{eqstar}) the estimate
\begin{align*}
\left\vert D^{\left(  2p\right)  }\widetilde{f}_{k,l}\left(  v\right)
\right\vert  &  \leq e^{2pv}\left\vert L_{k}^{p}\left(  f_{k,l}\right)
\left(  e^{v}\right)  \right\vert \leq Ce^{2p\left(  v_{0}+\gamma\right)
}\left(  2p\right)  !\left(  \tau+\varepsilon_{1}\right)  ^{2p}\\
&  \leq C\left(  2p\right)  !\left[  e^{v_{0}}\left(  \tau+\varepsilon\right)
\right]  ^{2p}.
\end{align*}
Thus the assumptions of Proposition \ref{Propevenodd} are satisfied and the
theorem is proven.
\end{proof}

Here is the main result about the analytical extension in the present section.

\begin{theorem}
\label{ThmMain}Let $f:A\left(  r_{0},r_{1}\right)  \rightarrow\mathbb{C}$ be
polyharmonic of infinite order and type $\tau\geq0$ and define $\widetilde
{f}_{k,l}\left(  v\right)  :=f_{k,l}\left(  e^{v}\right)  $ for $v\in\left(
\log r_{0},\log r_{1}\right)  $. Then, given $v_{0}\in$ $\left(  \log
r_{0},\log r_{1}\right)  ,$ the series
\[
\sum_{n=0}^{\infty}D^{\left(  n\right)  }\widetilde{f}_{k,l}\left(
v_{0}\right)  \cdot\Phi_{n}\left(  v-v_{0}\right)
\]
defines an analytic extension of $\widetilde{f}_{k,l},$ and it converges
compactly and absolutely in the disk with center $v_{0}$ and radius
\[
\ln\left(  1+\frac{1}{e^{v_{0}}\cdot\tau}\right)  .
\]
If $f$ is polyharmonic of infinite order and type $0,$ then $\widetilde
{f}_{k,l}\left(  v\right)  $ is an entire function and the Fourier-Laplace
coefficient $f_{k,l}$ possesses an analytic extension to the cutted complex
plane $\mathbb{C}_{-}:=\mathbb{C}\setminus\left(  -\infty,0\right]  $.
\end{theorem}

\begin{proof}
Theorems \ref{ThmGutest} and \ref{ThmTaylor2} show the first statement. If $f$
is polyharmonic of infinite order and type $0,$ the convergence radius is
infinite and $\widetilde{f}_{k,l}$ is entire. Now define $g\left(  z\right)
=\widetilde{f}_{k,l}\left(  \log z\right)  $ for all $z$ in the cutted complex
plane $\mathbb{C}_{-}.$ Then for $r\in\left(  r_{0},r_{1}\right)  $ we have
\[
g\left(  r\right)  =\widetilde{f}_{k,l}\left(  \log r\right)  =f_{k,l}\left(
e^{\log r}\right)  =f_{k,l}\left(  r\right)  .
\]

\end{proof}

\section{\label{S5}Analytic extensions of Fourier-Laplace coefficients for odd
dimension}

Assume that the dimension $d$ of the underlying euclidean space is odd. Then
for any fixed $k\in\mathbb{N}_{0}$ the exponents
\begin{equation}
\lambda_{2j}\left(  k\right)  :=k+2j\text{ and }\lambda_{2j+1}\left(
k\right)  :=-k-d+2+2j \label{eqLLnew}%
\end{equation}
defined in (\ref{eqlambda}) are pairwise different and $\left|  \lambda
_{m}\left(  k\right)  -\lambda_{n}\left(  k\right)  \right|  \geq1$ for all
$m\neq n.$ Since $\lambda_{n}\left(  k\right)  $ are pairwise different, the
defining equality (\ref{defPhi}) for the fundamental function $\Phi_{n}$
implies that
\begin{equation}
\Phi_{n}\left(  v\right)  =\sum_{j=0}^{n}\frac{e^{\lambda_{j}\left(  k\right)
v}}{q_{n}^{\prime}\left(  \lambda_{j}\left(  k\right)  \right)  }\text{ where
}q_{n}^{\prime}\left(  \lambda_{j}\left(  k\right)  \right)  =\prod
_{\substack{s=0 \\s\neq j}}^{n}\left(  \lambda_{j}\left(  k\right)
-\lambda_{s}\left(  k\right)  \right)  . \label{FiPrepres}%
\end{equation}
Here the polynomial
\[
q_{n}\left(  \lambda\right)  =\prod_{j=0}^{n}\left(  \lambda-\lambda
_{j}\right)
\]
is the symbol of the linear differential operator $L$ defined in
(\ref{eqdefL}) for which the notation $L\left(  \lambda\right)  $ would be
more traditional.

If $f$ is polyharmonic of infinite order and type $\tau$ and $v_{0}\in\left(
\log r_{0},\log r_{1}\right)  $ then, according to Theorem \ref{ThmMain}, the
series
\[
f_{k,l}\left(  e^{v}\right)  =\widetilde{f}_{k,l}\left(  v\right)  =\sum
_{n=0}^{\infty}D^{\left(  n\right)  }\widetilde{f}_{k,l}\left(  v_{0}\right)
\Phi_{n}\left(  v-v_{0}\right)
\]
converges for $v$ in a neighborhood of $v_{0}.$ It follows that
\[
f_{k,l}\left(  e^{v}\right)  =\sum_{n=0}^{\infty}\sum_{j=0}^{n}D^{\left(
n\right)  }\widetilde{f}_{k,l}\left(  v_{0}\right)  \frac{e^{\lambda
_{j}\left(  k\right)  \cdot\left(  v-v_{0}\right)  }}{q_{n}^{\prime}\left(
\lambda_{j}\left(  k\right)  \right)  }.
\]
Substituting $e^{v}=r$ back we arrive at
\[
f_{k,l}\left(  r\right)  =\sum_{n=0}^{\infty}\sum_{j=0}^{n}D^{\left(
n\right)  }\widetilde{f}_{k,l}\left(  v_{0}\right)  \frac{e^{-\lambda
_{j}\left(  k\right)  v_{0}}}{q_{n}^{\prime}\left(  \lambda_{j}\left(
k\right)  \right)  }r^{\lambda_{j}\left(  k\right)  }.
\]
In the following, we want to prove that this double series converges compactly
and absolutely, even for complex values $r,$ in the punctured plane
$\mathbb{C}^{\ast}$ provided that $f$ is polyharmonic of infinite order and
type $0.$

First we need an estimate for $\left\vert q_{n}^{\prime}\left(  \lambda
_{j}\right)  \right\vert $:

\begin{proposition}
\label{Propqn}Let $\lambda_{0},\ldots,\lambda_{n}$ be real numbers such that
$\left\vert \lambda_{s}-\lambda_{t}\right\vert \geq\alpha>0$ for all
$s,t\in\left\{  0,\ldots,n\right\}  ,s\neq t.$ Then for $q_{n}\left(
z\right)  =\left(  z-\lambda_{0}\right)  \cdots\left(  z-\lambda_{n}\right)  $
we have
\[
\left\vert q_{n}^{\prime}\left(  \lambda_{j}\right)  \right\vert
=\lim_{z\rightarrow\lambda_{j}}\left\vert \frac{q_{n}\left(  z\right)
}{z-\lambda_{j}}\right\vert \geq\frac{\alpha^{n}n!}{2^{n}}\text{ for all
}j=0,\ldots,n.
\]

\end{proposition}

\begin{proof}
We may assume that $\lambda_{0}<\cdots<\lambda_{n}.$ Then
\[
\lim_{z\rightarrow\lambda_{j}}\frac{q_{n}\left(  z\right)  }{\left(
z-\lambda_{j}\right)  }=\left(  \lambda_{j}-\lambda_{0}\right)  \cdots\left(
\lambda_{j}-\lambda_{j-1}\right)  \left(  \lambda_{j}-\lambda_{j+1}\right)
\cdots\left(  \lambda_{j}-\lambda_{n}\right)  .
\]
Using $\lambda_{0}<\cdots<\lambda_{n}$ we obtain an estimate for
$\lambda_{k+l}-\lambda_{k}$ as
\[
\lambda_{k+l}-\lambda_{k}=\lambda_{k+l}-\lambda_{k+l-1}+\lambda_{k+l-1}%
-\lambda_{k+l-2}+\lambda_{k+l+2}-\cdots+\lambda_{k+1}-\lambda_{k}\geq
l\cdot\alpha.
\]
Finally we obtain
\[
\lim_{z\rightarrow\lambda_{j}}\left\vert \frac{q_{n}\left(  z\right)
}{\left(  z-\lambda_{j}\right)  }\right\vert \geq\alpha^{n}j!\left(
n-j\right)  !=\frac{\alpha^{n}n!}{\binom{n}{j}}\geq\alpha^{n}n!\frac{1}{2^{n}%
}.
\]

\end{proof}

The next result strengthens Theorem \ref{ThmMain} for odd dimension $d>1$. For
example, in the case that $f$ is polyharmonic of infinite order and type $0,$
it follows that the Fourier-Laplace coefficients possess an analytic extension
to the punctured plane $\mathbb{C}^{\ast}$ instead of the cutted complex plane
$\mathbb{C}\setminus\left(  -\infty,0\right]  .$ In Theorem \ref{ThmCoeff}
below we shall give an explicit representation of the Fourier-Laplace
coefficients giving a proof of formula (\ref{ENEW}) mentioned in the introduction.

\begin{theorem}
\label{ThmMain1}Let $d>1$ be odd and $\lambda_{j}\left(  k\right)  $ as in
(\ref{eqLLnew}). Let $f:A\left(  r_{0},r_{1}\right)  \rightarrow\mathbb{C}$ be
polyharmonic of infinite order and type $\tau<1/2r_{0}.$ Then for any $v_{0}$
with $r_{0}<e^{v_{0}}<\min\left\{  r_{1},1/2\tau\right\}  ,$ the series
\begin{equation}
F_{k,l}\left(  z\right)  :=\sum_{n=0}^{\infty}\sum_{j=0}^{n}D^{\left(
n\right)  }\widetilde{f}_{k,l}\left(  v_{0}\right)  \frac{e^{-\lambda
_{j}\left(  k\right)  v_{0}}}{q_{n}^{\prime}\left(  \lambda_{j}\left(
k\right)  \right)  }z^{\lambda_{j}\left(  k\right)  } \label{eqseriesF}%
\end{equation}
converges compactly and absolutely in the annulus $\left\{  z\in
\mathbb{C};0<\left\vert z\right\vert <1/2\tau\right\}  $ and $F_{k,l}\left(
r\right)  =f_{k,l}\left(  r\right)  $ for all $r\in\left(  r_{0},\min\left\{
r_{1},1/2\tau\right\}  \right)  $.
\end{theorem}

\begin{proof}
Let $K$ be a compact subset of $\left\{  z\in\mathbb{C};0<\left\vert
z\right\vert <1/2\tau\right\}  .$ Then there exists $\rho\in\left(
0,1/2\tau\right)  $ with $K\subset\left\{  z\in\mathbb{C};0<\left\vert
z\right\vert \leq\rho\right\}  .$ Let $v_{0}$ satisfy $r_{0}<e^{v_{0}}%
<\min\left\{  r_{1},1/2\tau\right\}  $. If necessary, we can make $\rho$
larger such that
\begin{equation}
e^{v_{0}}<\rho<1/2\tau. \label{eqrho1}%
\end{equation}
Then $2\rho\tau<1$ and therefore there exists $\varepsilon>0$ such that
$2\rho\left(  \tau+\varepsilon\right)  <1.$ Theorem \ref{ThmGutest} provides
the estimate
\[
\left\vert D^{\left(  n\right)  }\widetilde{f}_{k,l}\left(  v_{0}\right)
\right\vert \leq Cn!e^{nv_{0}}\left(  \tau+\varepsilon\right)  ^{n}.
\]
Since $1/\left\vert q_{n}^{\prime}\left(  \lambda_{j}\right)  \right\vert
\leq2^{n}/n!,$ we obtain
\[
A\left(  z\right)  :=\sum_{n=0}^{\infty}\sum_{j=0}^{n}\left\vert D^{\left(
n\right)  }\widetilde{f}_{k,l}\left(  v_{0}\right)  \frac{e^{-\lambda
_{j}\left(  k\right)  v_{0}}}{q_{n}^{\prime}\left(  \lambda_{j}\left(
k\right)  \right)  }z^{\lambda_{j}\left(  k\right)  }\right\vert \leq
\sum_{n=0}^{\infty}Ce^{nv_{0}}2^{n}\left(  \tau+\varepsilon\right)  ^{n}%
\sum_{j=0}^{n}\left\vert ze^{-v_{0}}\right\vert ^{\lambda_{j}\left(  k\right)
}.
\]
Moreover
\[
\sum_{j=0}^{n}\left\vert ze^{-v_{0}}\right\vert ^{\lambda_{j}\left(  k\right)
}\leq\left(  \left\vert ze^{-v_{0}}\right\vert ^{k}+\left\vert ze^{-v_{0}%
}\right\vert ^{-k-d+2}\right)  \sum_{j=0}^{\left[  n/2\right]  }\left\vert
ze^{-v_{0}}\right\vert ^{2j}.
\]
Clearly $\left\vert ze^{-v_{0}}\right\vert ^{2j}\leq\left(  \rho e^{-v_{0}%
}\right)  ^{2j}$ for $\left\vert z\right\vert \leq\rho$ and for $j=0,\ldots
,\left[  n/2\right]  $. Since $\rho e^{-v_{0}}>1$ by (\ref{eqrho1}), we
estimate $\left(  \rho e^{-v_{0}}\right)  ^{2j}\leq\left(  \rho e^{-v_{0}%
}\right)  ^{n}$ and we obtain
\[
A\left(  z\right)  \leq\left(  \left\vert ze^{-v_{0}}\right\vert
^{k}+\left\vert ze^{-v_{0}}\right\vert ^{-k-d+2}\right)  \sum_{n=0}^{\infty
}C\left(  n+1\right)  \left(  2\left(  \tau+\varepsilon\right)  \rho\right)
^{n}%
\]
for all $z\in K.$ This series converges since $2\rho\left(  \tau
+\varepsilon\right)  <1.$
\end{proof}

\begin{theorem}
\label{ThmCoeff}Let $d>1$ be odd and $\lambda_{j}\left(  k\right)  $ as in
(\ref{eqLLnew}). Let $f:A\left(  r_{0},r_{1}\right)  \rightarrow\mathbb{C}$ be
polyharmonic of infinite order and type $\tau<1/2r_{0}.$ Then for each
$k\in\mathbb{N}_{0},l=1,\ldots,a_{k},$ there exist complex numbers $a_{k,l,j}$
with $j\in\mathbb{N}_{0}$ such that
\begin{equation}
f_{k,l}\left(  z\right)  =z^{k}\sum_{j=0}^{\infty}a_{k,l,2j}z^{2j}%
+z^{-k-d+2}\sum_{j=0}^{\infty}a_{k,l,2j+1}z^{2j} \label{eqlaurent}%
\end{equation}
converges compactly and absolutely in the annulus $\left\{  z\in
\mathbb{C};0<\left\vert z\right\vert <1/2\tau\right\}  $. The power series
\[
f_{k,l}^{\left(  1\right)  }\left(  z\right)  :=\sum_{j=0}^{\infty}%
a_{k,l,2j}z^{2j}\text{ and }f_{k,l}^{\left(  2\right)  }\left(  z\right)
:=\sum_{j=0}^{\infty}a_{k,l,2j+1}z^{2j}%
\]
have convergence radius at least $1/2\tau.$
\end{theorem}

\begin{proof}
1. First we define the coefficients $a_{k,l,j}.$ Since $\tau<1/2r_{0},$ there
exists $v_{0}\in\left(  \log r_{0},\log\left(  1/2\tau\right)  \right)  $ and
we can assume that $e^{v_{0}}<r_{1}.$ Then $e^{v_{0}}<1/2\tau$ and we can find
$\varepsilon>0$ such that $e^{v_{0}}\left(  \tau+\varepsilon\right)  <1/2.$ We
put
\begin{equation}
a_{k,l,j}:=e^{-\lambda_{j}\left(  k\right)  v_{0}}\sum_{n=j}^{\infty}%
\frac{D^{\left(  n\right)  }\widetilde{f}_{k,l}\left(  v_{0}\right)  }%
{q_{n}^{\prime}\left(  \lambda_{j}\left(  k\right)  \right)  }.
\label{eqdefaaa}%
\end{equation}
Using the estimate (\ref{estest}) in Theorem \ref{ThmGutest}, we see that
\[
\sum_{n=j}^{\infty}\left\vert \frac{D^{\left(  n\right)  }\widetilde{f}%
_{k,l}\left(  v_{0}\right)  }{q_{n}^{\prime}\left(  \lambda_{j}\left(
k\right)  \right)  }\right\vert \leq C\sum_{n=j}^{\infty}\left[  e^{v_{0}%
}\left(  \tau+\varepsilon\right)  \right]  ^{n}\frac{n!}{\left\vert
q_{n}^{\prime}\left(  \lambda_{j}\left(  k\right)  \right)  \right\vert },
\]
and the last series is converging using the ratio test for $b_{n}%
:=n!/\left\vert q_{n}^{\prime}\left(  \lambda_{j}\left(  k\right)  \right)
\right\vert $
\[
\frac{b_{n+1}}{b_{n}}=\frac{n+1}{\left\vert \lambda_{n+1}\left(  k\right)
-\lambda_{j}\left(  k\right)  \right\vert }\rightarrow1
\]
and the fact that $e^{v_{0}}\left(  \tau+\varepsilon\right)  <1/2$. So far we
have proven that the coefficients $a_{k,l,j}$ are well defined.

2. Using (\ref{eqdefaaa}) and the fact that $1/\left\vert q_{n}^{\prime
}\left(  \lambda_{j}\left(  k\right)  \right)  \right\vert \leq2^{n}/n!,$ we
obtain
\begin{equation}
\left\vert a_{k,l,j}\right\vert \leq Ce^{-\lambda_{j}\left(  k\right)  v_{0}%
}\sum_{n=j}^{\infty}\left[  2e^{v_{0}}\left(  \tau+\varepsilon\right)
\right]  ^{n}=Ce^{-\lambda_{j}\left(  k\right)  v_{0}}\frac{\left[  2e^{v_{0}%
}\left(  \tau+\varepsilon\right)  \right]  ^{j}}{1-2e^{v_{0}}\left(
\tau+\varepsilon\right)  }. \label{eqestaa}%
\end{equation}
Using the definition of $\lambda_{2j}\left(  k\right)  $ and $\lambda
_{2j+1}\left(  k\right)  ,$ we obtain the estimate
\begin{align}
\left\vert a_{k,l,2j}\right\vert  &  \leq Ce^{-kv_{0}}\frac{\left[  2\left(
\tau+\varepsilon\right)  \right]  ^{2j}}{1-2e^{v_{0}}\left(  \tau
+\varepsilon\right)  },\label{eqaaaa1}\\
\left\vert a_{k,l,2j+1}\right\vert  &  \leq Ce^{\left(  k+d\right)  v_{0}%
}\frac{\left[  2\left(  \tau+\varepsilon\right)  \right]  ^{2j+1}}%
{1-2e^{v_{0}}\left(  \tau+\varepsilon\right)  }. \label{eqaaaa2}%
\end{align}
It follows that $\overline{\lim_{j\rightarrow\infty}}\sqrt[2j]{\left\vert
a_{k,l,2j}\right\vert }\leq2\left(  \tau+\varepsilon\right)  $ and
$\overline{\lim_{j\rightarrow\infty}}\sqrt[2j]{\left\vert a_{k,l,2j+1}%
\right\vert }\leq2\left(  \tau+\varepsilon\right)  $ for any $\varepsilon>0,$
from which we conclude that the power series $f_{k,l}^{\left(  1\right)  }$
and $f_{k,l}^{\left(  2\right)  }$ have convergence radius at least $1/2\tau.$

3. By Theorem \ref{ThmMain1} the series
\[
\sum_{n=0}^{\infty}\sum_{j=0}^{n}D^{\left(  n\right)  }\widetilde{f}%
_{k,l}\left(  v_{0}\right)  \frac{e^{-\lambda_{j}\left(  k\right)  v_{0}}%
}{q_{n}^{\prime}\left(  \lambda_{j}\left(  k\right)  \right)  }z^{\lambda
_{j}\left(  k\right)  }%
\]
converges compactly on each compact subset $K$ of $\left\{  z\in
\mathbb{C};0<\left\vert z\right\vert <1/2\tau\right\}  .$ So we may rearrange
the series and the series
\[
\sum_{j=0}^{\infty}z^{\lambda_{j}\left(  k\right)  }e^{-\lambda_{j}\left(
k\right)  v_{0}}\sum_{n=j}^{\infty}\frac{D^{\left(  n\right)  }\widetilde
{f}_{k,l}\left(  v_{0}\right)  }{q_{n}^{\prime}\left(  \lambda_{j}\left(
k\right)  \right)  }=\sum_{j=0}^{\infty}a_{k,l,j}z^{\lambda_{j}\left(
k\right)  }%
\]
converges compactly in $\left\{  z\in\mathbb{C};0<\left\vert z\right\vert
<1/2\tau\right\}  .$ The decomposition (\ref{eqlaurent}) follows by splitting
the sum over odd and even indices. The proof is complete.
\end{proof}

\begin{remark}
The coefficients $a_{k,l,j}$ do not depend on the special value $v_{0}$ since
the coefficients in (\ref{eqlaurent}) are unique. The coefficients $a_{k,l,j}$
in (\ref{eqdefaaa}) are well defined provided that $f$ is polyharmonic of
infinite order and type $<1/r_{0}.$ However, for the estimate (\ref{eqestaa})
we needed that the type $\tau$ is smaller than $1/2r_{0}$.
\end{remark}

\section{Analytic extensions of polyharmonic functions of infinite order for
odd dimension}

We recall some notations and basic facts. We have defined $q\left(  z\right)
:=z_{1}^{2}+\cdots+z_{d}^{2}$ for $z=\left(  z_{1},\ldots,z_{d}\right)
\in\mathbb{C}^{d}$ and clearly the following inequality holds for all
$z\in\mathbb{C}^{d}:$
\[
\left\vert q\left(  z\right)  \right\vert \leq\left\vert z_{1}\right\vert
^{2}+\cdots+\left\vert z_{d}\right\vert ^{2}=:\left\vert z\right\vert
_{\mathbb{C}^{d}}^{2}.
\]
Note that $q\left(  z\right)  $ is the analytic extension of $\left\vert
x\right\vert ^{2}=x_{1}^{2}+\cdots+x_{d}^{2}.$ The \emph{Lie norm}
$L_{+}\left(  z\right)  \in\left[  0,\infty\right)  $ is defined by the
equation
\[
L_{+}\left(  z\right)  ^{2}=\left\vert z\right\vert _{\mathbb{C}^{d}}%
^{2}+\sqrt{\left\vert z\right\vert _{\mathbb{C}^{d}}^{4}-\left\vert q\left(
z\right)  \right\vert ^{2}}\text{ for }z\in\mathbb{C}^{d}%
\]
(see e.g.\ \cite{ACL83}, \cite{Avan85}, \cite{Mori98}, \cite{Shai03}). Note
that $\left\vert z\right\vert _{\mathbb{C}^{d}}\leq L_{+}\left(  z\right)  $
for all $z\in\mathbb{C}^{d}.$ In \cite{KoRe08} the following estimate is
established (see also \cite{FuMo02}):
\begin{equation}
\sum_{l=1}^{a_{k}}\left\vert Y_{k,l}\left(  z\right)  \right\vert ^{2}%
\leq\frac{a_{k}}{\omega_{d-1}}\left(  \left\vert z\right\vert ^{2}%
+\sqrt{\left\vert z\right\vert ^{4}-\left\vert q\left(  z\right)  \right\vert
^{2}}\right)  ^{k}=\frac{a_{k}}{\omega_{d-1}}\left(  L_{+}\left(  z\right)
\right)  ^{2k} \label{ineqLie}%
\end{equation}
for all $z\in\mathbb{C}^{d}.$ Using the Cauchy Schwarz inequality one obtains
\begin{equation}
\sum_{l=1}^{a_{k}}\left\vert Y_{k,l}\left(  z\right)  \right\vert \leq
\sqrt{a_{k}}\sqrt{\sum_{l=1}^{a_{k}}\left\vert Y_{k,l}\left(  z\right)
\right\vert ^{2}}\leq\frac{a_{k}}{\sqrt{\omega_{d-1}}}\left(  L_{+}\left(
z\right)  \right)  ^{k}. \label{eqbasic}%
\end{equation}
Now we define $L_{-}\left(  z\right)  :=\sqrt{\left\vert z\right\vert
_{\mathbb{C}^{d}}^{2}-\sqrt{\left\vert z\right\vert _{\mathbb{C}^{d}}%
^{4}-\left\vert q\left(  z\right)  \right\vert ^{2}}}$ for $z\in\mathbb{C}%
^{d}.$ Then $0\leq L_{-}\left(  z\right)  \leq L_{+}\left(  z\right)  $ and it
is easy to see that
\[
L_{+}\left(  z\right)  L_{-}\left(  z\right)  =\left\vert q\left(  z\right)
\right\vert \text{ for all }z\in\mathbb{C}^{d}.
\]
In analogy to the Lie ball we define the \emph{Lie annulus} as the set
\[
\widetilde{A}\left(  r_{0},r_{1}\right)  :=\left\{  z\in\mathbb{C}^{d}%
;r_{0}<L_{-}\left(  z\right)  \text{ and }L_{+}\left(  z\right)
<r_{1}\right\}  .
\]
In \cite[p.\ 95]{Avan85} it is shown that $\widetilde{A}\left(  r_{0}%
,r_{1}\right)  $ is the harmonicity hull of the annular domain $A\left(
r_{0},r_{1}\right)  .$ It can be shown that $\widetilde{A}\left(  r_{0}%
,r_{1}\right)  $ is connected. On the other hand, the complement of
$\widetilde{A}\left(  r_{0},r_{1}\right)  $ in $\mathbb{C}^{d}$ is connected
as well, in contrast to the fact that the complement of the annular region
$A\left(  r_{0},r_{1}\right)  $ in $\mathbb{R}^{d}$ consists of two connected components.

It is known that a polyharmonic function of infinite order and type $0$ can be
extended to a \emph{multi-valued analytic} function on the harmonicity hull
(see \cite{ACL83}) and to a \emph{single-valued analytic} function on
$\ker(\widetilde{A}\left(  r_{0},r_{1}\right)  ),$ the \emph{kernel} (in
\cite[p.\ 131]{Avan85} \emph{noyau}) of the harmonicity hull (see
\cite[p.\ 135]{Avan85} for details) which is clearly contained in the set
$\widetilde{A}\left(  r_{0},r_{1}\right)  \setminus q^{-1}\left(  \left(
-\infty,0\right]  \right)  .$

We present now our main result about analytical extendibility of polyharmonic
functions of infinite order and type $\tau<1/2r_{1}$ on the annular region
$A\left(  r_{0},r_{1}\right)  $.

\begin{theorem}
\label{ThmMM}Let $d>1$ be odd and let $f:A\left(  r_{0},r_{1}\right)
\rightarrow\mathbb{C}$ be polyharmonic of infinite order and type
$\tau<1/2r_{1}.$ Then there exist analytic functions
\begin{align*}
f_{1}  &  :\left\{  z\in\mathbb{C}^{d};L_{+}\left(  z\right)  <r_{1}\right\}
\rightarrow\mathbb{C}\\
f_{2}  &  :\{z\in\mathbb{C}^{d};r_{0}<L_{-}\left(  z\right)  \leq L_{+}\left(
z\right)  <1/2\tau\}\rightarrow\mathbb{C}%
\end{align*}
such that
\[
F\left(  z\right)  =f_{1}\left(  z\right)  +\left(  z_{1}^{2}+\cdots+z_{d}%
^{2}\right)  ^{\left(  2-d\right)  /2}f_{2}\left(  z\right)
\]
is an analytic extension of $f.$ Here $F$ is defined for all $z\in$
$\widetilde{A}\left(  r_{0},r_{1}\right)  \setminus q^{-1}\left(  \left(
-\infty,0\right]  \right)  .$
\end{theorem}

\begin{proof}
1. Let us recall that $f\left(  x\right)  =\sum_{k=0}^{\infty}\sum
_{l=1}^{a_{k}}f_{k,l}\left(  r\right)  Y_{k,l}\left(  \theta\right)  $ for
$x=r\theta$, and let $\lambda_{j}\left(  k\right)  $ be as in (\ref{eqLLnew}).
By Theorem \ref{ThmCoeff}, each Fourier-Laplace coefficient $f_{k,l}\left(
r\right)  $ can be expanded in a series of type $\sum_{j=0}^{\infty}%
a_{k,l,j}r^{\lambda_{j}\left(  k\right)  },$ and hence
\begin{equation}
f\left(  x\right)  =\sum_{k=0}^{\infty}\sum_{l=1}^{a_{k}}\sum_{j=0}^{\infty
}a_{k,l,j}r^{\lambda_{j}\left(  k\right)  }Y_{k,l}\left(  \theta\right)  .
\label{eqfff}%
\end{equation}
Moreover $Y_{k,l}\left(  x\right)  =r^{k}Y_{k,l}\left(  \theta\right)  $ for
$x=r\theta.$ We consider even and odd indices $j$ in (\ref{eqfff}) and define
two functions $f_{1}$ and $f_{2}$ such that $f\left(  x\right)  =f_{1}\left(
x\right)  +r^{2-d}f_{2}\left(  x\right)  ,$ where
\begin{align*}
f_{1}\left(  x\right)   &  =\sum_{k=0}^{\infty}\sum_{l=1}^{a_{k}}\sum
_{j=0}^{\infty}a_{k,l,2j}r^{2j}Y_{k,l}\left(  x\right)  ,\\
f_{2}\left(  x\right)   &  =\sum_{k=0}^{\infty}\sum_{l=1}^{a_{k}}\sum
_{j=0}^{\infty}a_{k,l,2j+1}r^{-2k+2j}Y_{k,l}\left(  x\right)  .
\end{align*}
We shall show that $f_{1}\left(  x\right)  $ can be analytically extended for
all $z$ with $L_{+}\left(  z\right)  <r_{1}$ and that $f_{2}\left(  x\right)
$ can be analytically extended for all $z$ with $r_{0}<L_{-}\left(  z\right)
$ and $L_{+}\left(  z\right)  <1/2\tau.$

2. The function $r^{2}=x_{1}^{2}+\cdots+x_{d}^{2}$ has the analytic extension
$q\left(  z\right)  =z_{1}^{2}+\cdots+z_{d}^{2}$ for $z=\left(  z_{1}%
,\ldots,z_{d}\right)  \in\mathbb{C}^{d}.$ The polynomial $Y_{k,l}\left(
x\right)  $ has the analytic extension $Y_{k,l}\left(  z\right)  .$

Next we show that
\[
F_{1}\left(  z\right)  :=\sum_{k=0}^{\infty}\sum_{l=1}^{a_{k}}\sum
_{j=0}^{\infty}a_{k,l,2j}q\left(  z\right)  ^{j}Y_{k,l}\left(  z\right)
\]
converges absolutely for all $z$ with $L_{+}\left(  z\right)  \leq\rho$ for
any $0<\rho<r_{1}.$ Since $\rho<r_{1}$ and $\tau<1/2r_{1},$ we can find
$v_{0}\in\left(  \log r_{0},\log r_{1}\right)  $ such that $\rho<e^{v_{0}%
}<1/2\tau.$ Choose $\varepsilon>0$ such that $2e^{v_{0}}\left(  \tau
+\varepsilon\right)  <1.$ We use now (\ref{eqaaaa1}) and the estimate
$\left\vert q\left(  z\right)  \right\vert \leq\left\vert z\right\vert
_{\mathbb{C}^{d}}^{2}\leq L_{+}^{2}\left(  z\right)  \leq\rho^{2}$ and we
obtain
\[
\left\vert F_{1}\left(  z\right)  \right\vert \leq C\sum_{k=0}^{\infty}%
\sum_{l=1}^{a_{k}}\sum_{j=0}^{\infty}\left\vert Y_{k,l}\left(  z\right)
\right\vert e^{-kv_{0}}\frac{\left[  2\left(  \tau+\varepsilon\right)
\right]  ^{2j}}{1-2e^{v_{0}}\left(  \tau+\varepsilon\right)  }\rho^{2j}.
\]

Since $\rho e^{-v_{0}}<1$ and $2e^{v_{0}}\left(  \tau+\varepsilon\right)  <1,$
the series $\sum_{j=0}^{\infty}\left[  2\left(  \tau+\varepsilon\right)
\rho\right]  ^{2j}$ converges and there exists a constant $C_{1}$ such that
\[
\left\vert F_{1}\left(  z\right)  \right\vert \leq C_{1}\sum_{k=0}^{\infty
}\sum_{l=1}^{a_{k}}\left\vert Y_{k,l}\left(  z\right)  \right\vert e^{-kv_{0}%
}\leq C_{1}\sum_{k=0}^{\infty}e^{-kv_{0}}\frac{a_{k}}{\sqrt{\omega_{d-1}}%
}\left(  L_{+}\left(  z\right)  \right)  ^{k}%
\]
where we have used (\ref{eqbasic}). Since $L_{+}\left(  z\right)  \leq\rho$
and $\rho e^{-v_{0}}<1,$ we see that the last sum converges.

3. It remains to show that
\[
F_{2}\left(  z\right)  =\sum_{k=0}^{\infty}\sum_{l=1}^{a_{k}}\sum
_{j=0}^{\infty}a_{k,l,2j+1}q\left(  z\right)  ^{-k+j}Y_{k,l}\left(  z\right)
\]
converges compactly in $\left\{  r_{0}<L_{-}\left(  z\right)  \text{ and
}L_{+}\left(  z\right)  <1/2\tau\right\}  .$ Let $\rho_{0}$ and $\rho_{1}$ be
positive numbers such that $r_{0}<\rho_{0}$ and $\rho_{1}<1/2\tau,$ and assume
that $L_{-}\left(  z\right)  \geq\rho_{0}$ and $L_{+}\left(  z\right)
\leq\rho_{1}.$ Choose $v_{0}$ such that $r_{0}<e^{v_{0}}<\rho_{0}$, so
$e^{v_{0}}/\rho_{0}<1$. Moreover we can assume that $e^{v_{0}}<1/2\tau$ in
view of our general assumption $r_{0}<1/2\tau.$ Then there exists
$\varepsilon>0$ such that
\begin{equation}
2e^{v_{0}}\left(  \tau+\varepsilon\right)  <1\text{ and }2\left(
\tau+\varepsilon\right)  \rho_{1}<1. \label{eqnice}%
\end{equation}
The estimate (\ref{eqaaaa2}) gives
\[
\left\vert F_{2}\left(  z\right)  \right\vert \leq C\sum_{k=0}^{\infty}%
\sum_{l=1}^{a_{k}}\sum_{j=0}^{\infty}e^{\left(  k+d\right)  v_{0}}%
\frac{\left[  2\left(  \tau+\varepsilon\right)  \right]  ^{2j+1}}{1-2e^{v_{0}%
}\left(  \tau+\varepsilon\right)  }\left\vert q\left(  z\right)  \right\vert
^{-k+j}\left\vert Y_{k,l}\left(  z\right)  \right\vert .
\]
Since $\left\vert q\left(  z\right)  \right\vert \leq L_{+}^{2}\left(
z\right)  \leq\rho_{1}^{2},$ we estimate
\[
\left\vert F_{2}\left(  z\right)  \right\vert \leq C\sum_{k=0}^{\infty}%
\sum_{l=1}^{a_{k}}e^{\left(  k+d\right)  v_{0}}\left\vert q\left(  z\right)
\right\vert ^{-k}\left\vert Y_{k,l}\left(  z\right)  \right\vert \sum
_{j=0}^{\infty}\frac{1}{\rho_{1}}\frac{\left[  2\left(  \tau+\varepsilon
\right)  \rho_{1}\right]  ^{2j+1}}{1-2e^{v_{0}}\left(  \tau+\varepsilon
\right)  }.
\]

The last series converges since $2\left(  \tau+\varepsilon\right)  \rho
_{1}<1,$ and is bounded by a constant, say $C_{1}.$ Further (\ref{eqbasic})
implies
\[
\left\vert F_{2}\left(  z\right)  \right\vert \leq CC_{1}\sum_{k=0}^{\infty
}e^{\left(  k+d\right)  v_{0}}\frac{a_{k}}{\sqrt{\omega_{d-1}}}\frac{\left(
L_{+}\left(  z\right)  \right)  ^{k}}{\left\vert q\left(  z\right)
\right\vert ^{k}}.
\]
Recall that $L_{+}\left(  z\right)  L_{-}\left(  z\right)  =\left\vert
q\left(  z\right)  \right\vert $ , so we can estimate $L_{+}\left(  z\right)
/\left\vert q\left(  z\right)  \right\vert =1/L_{-}\left(  z\right)
\leq1/\rho_{0}.$ Hence
\[
\left\vert F_{2}\left(  z\right)  \right\vert \leq CC_{1}\sum_{k=0}^{\infty
}e^{dv_{0}}\frac{a_{k}}{\sqrt{\omega_{d-1}}}\left(  \frac{e^{v_{0}}}{\rho_{0}%
}\right)  ^{k}%
\]
and this series converges since $e^{v_{0}}/\rho_{0}<1.$ The proof is complete.
\end{proof}

Let us illustrate the theorem for the case of a harmonic function $f$ on the
annular region $A\left(  r_{0},r_{1}\right)  .$ Then $f$ is of type $0$ and
the conclusion is that we can decompose $f$ as a sum $f_{1}\left(  x\right)
+\left\vert x\right\vert ^{2-d}f_{2}\left(  x\right)  ,$ where $f_{1}\left(
z\right)  $ is analytic on the Lie ball and $f_{2}\left(  z\right)  $ is
analytic for all $z$ with $L_{-}\left(  z\right)  >r_{0}.$

\section{\label{S7}Analytic extensions of polyharmonic functions of infinite
order for even dimension}

In this section we assume that the dimension $d$ is an even number. Let
$k\in\mathbb{N}_{0}$ be fixed. Then the exponents $\lambda_{2j}\left(
k\right)  =k+2j$ and $\lambda_{2j+1}\left(  k\right)  =-k-d+2+2j$ may be equal
and the description of the exponential space $E_{\left(  \lambda_{0}%
,\ldots,\lambda_{n}\right)  }$ defined in (\ref{Espace}) is different from the
odd case. Clearly $e_{k,2j}\left(  v\right)  :=e^{\lambda_{2j}\left(
k\right)  v}$ are solutions and for odd indices $2j+1$ we obtain the
solutions
\begin{align*}
e_{k,2j+1}\left(  v\right)   &  :=e^{\lambda_{2j+1}\left(  k\right)  v}\text{
if }{\lambda_{2j+1}}\left(  k\right)  {\notin}\left\{  k+2l;l\in\mathbb{N}%
_{0}\right\} \\
e_{k,2j+1}\left(  v\right)   &  :=v\cdot e^{\lambda_{2j+1}\left(  k\right)
v}\text{ if }{\lambda_{2j+1}}\left(  k\right)  {\in}\left\{  k+2l;l\in
\mathbb{N}_{0}\right\}
\end{align*}
Then, for suitable coefficients $d_{j},$ the fundamental function $\Phi
_{n}\left(  v\right)  $ is an exponential polynomial of the form
\begin{equation}
\Phi_{n}\left(  v\right)  =\sum_{j=0}^{n}d_{j}e_{k,j}\left(  v\right)
=\frac{1}{2\pi i}\int_{\Gamma_{r}}\frac{e^{vz}}{\left(  z-\lambda_{0}\right)
\cdots\left(  z-\lambda_{n}\right)  }dz. \label{eqphidec}%
\end{equation}
If $f$ is polyharmonic of infinite order and type $\tau$ and $v_{0}\in\left(
\log r_{0},\log r_{1}\right)  $ then, using Theorem \ref{ThmMain}, the series
\begin{equation}
f_{k,l}\left(  e^{v}\right)  =\widetilde{f}_{k,l}\left(  v\right)  =\sum
_{n=0}^{\infty}\sum_{j=0}^{n}D^{\left(  n\right)  }\widetilde{f}_{k,l}\left(
v_{0}\right)  \cdot d_{j}\cdot e_{k,j}\left(  v-v_{0}\right)  \label{eqdjdj}%
\end{equation}
converges for $v$ in a neighborhood of $v_{0}.$ Now one may try to formulate
results analogous to those given in Sections 5 and 6 where the system
$e^{\lambda_{j}\left(  k\right)  v}$ for $j\in\mathbb{N}_{0}$ is now replaced
by $e_{k,j}\left(  v\right)  $ for $j\in\mathbb{N}_{0}$. A quick inspection of
the proofs shows that one only needs an appropriate estimate for the
coefficients $d_{j}$ which in the odd case have been equal to $1/q_{n}%
^{\prime}\left(  \lambda_{j}\right)  $ for $j=0,\ldots,n.$ Below we shall show
that the coefficients $d_{j}$ in (\ref{eqdjdj}) satisfy the estimate
\begin{equation}
\left\vert d_{j}\right\vert \leq\frac{2^{n}}{\left(  n-2\right)  !},
\label{eqddd}%
\end{equation}
which is a little bit weaker than in the odd case but still good enough to
prove convergence of the involved sums. We shall leave the details to the
reader and formulate only one result for the even case:

\begin{theorem}
\label{ThmMM2}Let $d$ be even and let $f:A\left(  r_{0},r_{1}\right)
\rightarrow\mathbb{C}$ be polyharmonic of infinite order and type
$\tau<1/2r_{1}.$ Then there exist analytic functions
\begin{align*}
f_{1}  &  :\{z\in\mathbb{C}^{d};L_{+}\left(  z\right)  <r_{1}\}\rightarrow
\mathbb{C},\\
f_{2}  &  :\{z\in\mathbb{C}^{d};r_{0}<L_{-}\left(  z\right)  \text{ and }%
L_{+}\left(  z\right)  <1/2\tau\text{ and }q\left(  z\right)  \notin\left(
-\infty,0\right]  \}\rightarrow\mathbb{C}%
\end{align*}
such that
\[
F\left(  z\right)  =f_{1}\left(  z\right)  +\left(  z_{1}^{2}+\cdots+z_{d}%
^{2}\right)  ^{\left(  2-d\right)  /2}f_{2}\left(  z\right)
\]
is an analytic extension of $f.$
\end{theorem}

Now we proceed to the estimate of the coefficients $d_{j}.$ They can be
computed by the partial fraction decomposition of the integrand in formula
(\ref{eqphidec}). The next result addresses this problem:

\begin{proposition}
Let $\mu_{0},\ldots,\mu_{s}$ and $\nu_{1},\ldots,\nu_{r}$ be distinct real
numbers, $n:=2r+s+1,$ and define
\[
Q_{n}\left(  z\right)  =\left(  z-\mu_{0}\right)  \cdots\left(  z-\mu
_{s}\right)  \left(  z-\nu_{1}\right)  ^{2}\cdots\left(  z-\nu_{r}\right)
^{2}.
\]
Then the coefficients $a_{j}$ and $c_{j}$ in the partial fraction
decomposition
\[
\frac{1}{Q_{n}\left(  z\right)  }=\sum_{j=0}^{s}\frac{a_{j}}{z-\mu_{j}}%
+\sum_{j=1}^{r}\frac{b_{j}}{z-\nu_{j}}+\sum_{j=1}^{r}\frac{c_{j}}{\left(
z-\nu_{j}\right)  ^{2}}%
\]
are non-zero and are given by
\[
a_{j_{0}}=\lim_{z\rightarrow\mu_{j_{0}}}\frac{z-\mu_{j_{0}}}{Q_{n}\left(
z\right)  }\text{ and }c_{j_{0}}=\lim_{z\rightarrow v_{j_{0}}}\frac{\left(
z-\nu_{j_{0}}\right)  ^{2}}{Q_{n}\left(  z\right)  }.
\]
If $\left\vert \nu_{j_{0}}-\mu_{j}\right\vert \geq2$ for all $j=1,\ldots s,$
and $\left\vert \nu_{j_{0}}-\nu_{j}\right\vert \geq2$ for all $j=1,\ldots,r$
with $j\neq j_{0}$ then
\[
\left\vert b_{j_{0}}\right\vert \leq\left(  n-1\right)  \left\vert c_{j_{0}%
}\right\vert .
\]

\end{proposition}

\begin{proof}
It is easy to see that $1=\lim_{z\rightarrow\mu_{j_{0}}}a_{j_{0}}Q_{n}\left(
z\right)  /\left(  z-\mu_{j_{0}}\right)  $ and $1=\lim_{z\rightarrow\nu
_{j_{0}}}c_{j_{0}}Q_{n}\left(  z\right)  /\left(  z-\nu_{j_{0}}\right)  ^{2}.$
Further $b_{j_{0}}$ can be computed by residue theory:
\begin{equation}
b_{j_{0}}=\text{res}_{\nu_{j_{0}}}\frac{1}{Q_{n}\left(  z\right)  }=\frac
{d}{dz}\frac{\left(  z-\nu_{j_{0}}\right)  ^{2}}{Q_{n}\left(  z\right)
}\left(  \nu_{j_{0}}\right)  . \label{eqres}%
\end{equation}
Let us define $P_{j_{0}}\left(  z\right)  =Q_{n}\left(  z\right)  /\left(
z-\nu_{j_{0}}\right)  ^{2}.$ Clearly
\[
P_{j_{0}}\left(  \nu_{j_{0}}\right)  =\lim_{z\rightarrow v_{j_{0}}}\frac
{Q_{n}\left(  z\right)  }{\left(  z-\nu_{j_{0}}\right)  ^{2}}=\frac
{1}{c_{j_{0}}}.
\]
Then (\ref{eqres}) is equivalent to
\[
b_{j_{0}}=\frac{d}{dz}\frac{1}{P_{j_{0}}}\left(  \nu_{j_{0}}\right)  =-\left[
P_{j_{0}}\left(  \nu_{j_{0}}\right)  \right]  ^{-2}P_{j_{0}}^{\prime}\left(
\nu_{j_{0}}\right)  =-\frac{1}{P_{j_{0}}\left(  \nu_{j_{0}}\right)  }%
\frac{P_{j_{0}}^{\prime}\left(  \nu_{j_{0}}\right)  }{P_{j_{0}}\left(
\nu_{j_{0}}\right)  }.
\]
Moreover
\[
\frac{P_{j_{0}}^{\prime}\left(  \nu_{j_{0}}\right)  }{P_{j_{0}}\left(
\nu_{j_{0}}\right)  }=\sum_{j=0}^{s}\frac{1}{\nu_{j_{0}}-\mu_{j}}+2\sum_{j\neq
j_{0}}^{r}\frac{1}{\nu_{j_{0}}-\nu_{j}}.
\]
Since $\left\vert \nu_{j_{0}}-\mu_{j}\right\vert \geq2$ and $\left\vert
\nu_{j_{0}}-\nu_{j}\right\vert \geq2,$ we see that
\[
\left\vert \frac{P_{j_{0}}^{\prime}\left(  \nu_{j_{0}}\right)  }{P_{j_{0}%
}\left(  \nu_{j_{0}}\right)  }\right\vert \leq\frac{1}{2}\left(  s+1\right)
+r=\frac{1}{2}\left(  s+1+2r\right)  =\frac{1}{2}n\leq n-1.
\]
It follows that
\[
\left\vert b_{j_{0}}\right\vert \leq\frac{n-1}{\left\vert P_{j_{0}}\left(
\nu_{j_{0}}\right)  \right\vert }=\left(  n-1\right)  \left\vert c_{j_{0}%
}\right\vert .
\]
The proof is complete.
\end{proof}

We can order the real and distinct numbers $\mu_{0},\ldots,\mu_{s},\nu
_{1},\ldots,\nu_{r}$ according to their values, say $\lambda_{0}%
<\cdots<\lambda_{s+r}.$ Clearly $\left\vert \lambda_{j}-\lambda_{k}\right\vert
\geq\alpha:=2$ for all $k\neq j.$ The exponents $\lambda_{j}$ have either
multiplicity $m_{j}=1$ or $m_{j}=2,$ and $m_{0}+\cdots+m_{s+r}=n.$ Then
\[
\lim_{z\rightarrow\lambda_{j}}\frac{Q_{n}\left(  z\right)  }{\left(
z-\lambda_{j}\right)  ^{m_{j}}}=\left(  \lambda_{j}-\lambda_{0}\right)
^{m_{0}}\cdots\left(  \lambda_{j}-\lambda_{j-1}\right)  ^{m_{j-1}}\left(
\lambda_{j}-\lambda_{j+1}\right)  ^{m_{j+1}}\cdots\left(  \lambda_{j}%
-\lambda_{s+r}\right)  ^{m_{s+r}}.
\]
The proof of Proposition \ref{Propqn} shows that $\lambda_{k+l}-\lambda
_{k}\geq l\cdot\alpha.$ Thus
\[
\lim_{z\rightarrow\lambda_{j}}\left\vert \frac{Q_{n}\left(  z\right)
}{\left(  z-\lambda_{j}\right)  ^{m_{j}}}\right\vert \geq\alpha^{m_{0}%
}j^{m_{0}}\cdots\alpha^{m_{j-1}}1^{m_{j-1}}\cdot\alpha^{m_{j+1}}1^{m_{j+1}%
}\cdots\alpha^{m_{s+r}}\left(  s+r-j\right)  ^{m_{s+r}}=:A_{j}.
\]
The factor $\alpha=2$ occurs $n-m_{j}$ times in the integer $A_{j},$ which is
greater or equal than $n-2.$ Since $m_{l}\geq1$ for $l=0,1,\ldots,s+r,$ we
have clearly a factor
\[
j!\left(  s+r-j\right)  !=j!\left(  n-r-1-j\right)  !\leq j!\left(
n-r-j\right)  !
\]
in the expression $A_{j}$. Further $m_{l}=2$ for at least $r-1$ different
factors in $A_{j}$ which are non-zero integers; so the product of these number
is bigger or equal to $l!\left(  r-1-l\right)  !$ for some natural number
$l\in\left\{  1,\ldots,r-1\right\}  .$ Thus we conclude that
\[
\left\vert A_{j}\right\vert \geq2^{n-2}j!\left(  n-r-j\right)  !l!\left(
r-1-l\right)  !=2^{n-2}\frac{\left(  n-r\right)  !}{\binom{n-r}{j}}%
\frac{\left(  r-1\right)  !}{\binom{r-1}{l}}.
\]
Since $\binom{n-r}{j}\leq2^{n-r}$ and $\binom{r-1}{l}\leq2^{r-1},$ we obtain
\[
\left\vert A_{j}\right\vert \geq2^{-1}\left(  n-r\right)  !\left(  r-1\right)
!=2^{-1}\frac{\left(  n-1\right)  !}{\binom{n-1}{r-1}}\leq\frac{\left(
n-1\right)  !}{2^{n}}.
\]
It follows that $\left\vert a_{j}\right\vert \leq2^{n}/\left(  n-1\right)  !$
and $\left\vert c_{j}\right\vert \leq2^{n}/\left(  n-1\right)  !$. Further
$\left\vert b_{j}\right\vert \leq\left(  n-1\right)  \left\vert c_{j}%
\right\vert $ and we conclude that $\left\vert d_{j}\right\vert \leq
2^{n}/\left(  n-2\right)  !.$

\section{Appendix: Estimate of derivatives of odd order}

In this section we collect and prove some results about linear differential
operators which are needed in the paper. The following version of Rolle's
theorem is well known and the short proof is included for convenience of the reader.

\begin{theorem}
[Rolle's Theorem]\label{ThmRolle} Let $f:\left[  a,b\right]  \rightarrow
\mathbb{R}$ be continuous and differentiable on $\left(  a,b\right)  .$ For
$\lambda\in\mathbb{R}$ define the differential operator $D_{\lambda
}f:=df/dx-\lambda f.$ If $e^{-\lambda a}f\left(  a\right)  =e^{-\lambda
b}f\left(  b\right)  ,$ then there exists $\xi\in\left(  a,b\right)  $ with
$D_{\lambda}f\left(  \xi\right)  =0.$
\end{theorem}

\begin{proof}
Define $g\left(  x\right)  :=e^{-\lambda x}f\left(  x\right)  .$ Then
$g\left(  a\right)  =g\left(  b\right)  .$ By the theorem of Rolle there
exists $\xi\in\left(  a,b\right)  $ with $g^{\prime}\left(  \xi\right)  =0.$
Now $g^{\prime}\left(  x\right)  =f^{\prime}\left(  x\right)  e^{-\lambda
x}-\lambda f\left(  x\right)  e^{-\lambda x}=e^{-\lambda x}D_{\lambda}f\left(
x\right)  .$ So $g^{\prime}\left(  \xi\right)  =0$ implies that $D_{\lambda
}f\left(  \xi\right)  =0.$
\end{proof}

The following result is an analog of the mean value theorem which is indeed a
consequence when we let $\lambda$ go to $0.$

\begin{theorem}
\label{ThmIntermed}Let $f:\left[  a,b\right]  \rightarrow\mathbb{R}$ be
continuous on $\left[  a,b\right]  $ and differentiable on $\left(
a,b\right)  ,$ and let $\lambda\neq0$ be a real number. Then there exists
$\xi\in\left(  a,b\right)  $ with
\[
\frac{e^{\lambda a}f\left(  b\right)  -e^{\lambda b}f\left(  a\right)
}{e^{\lambda b}-e^{\lambda a}}\lambda=D_{\lambda}f\,\left(  \xi\right)  .
\]

\end{theorem}

\begin{proof}
Define for $\lambda\neq0$ a function $\psi$ by
\[
\psi\left(  x\right)  =f\left(  x\right)  -\frac{f\left(  b\right)
-e^{\lambda\left(  b-a\right)  }f\left(  a\right)  }{e^{\lambda b}-e^{\lambda
a}}\left(  e^{\lambda x}-e^{\lambda a}\right)  .
\]
Then $\psi\left(  a\right)  =f\left(  a\right)  $ and $\psi\left(  b\right)
=e^{\lambda\left(  b-a\right)  }f\left(  a\right)  ,$ so $e^{-\lambda b}%
\psi\left(  b\right)  =e^{-\lambda a}f\left(  a\right)  =e^{-\lambda a}%
\psi\left(  a\right)  .$ By Theorem \ref{ThmRolle} there exists $\xi\in\left(
a,b\right)  $ with $D_{\lambda}\psi\left(  \xi\right)  =0.$ Note that
$D_{\lambda}\left(  e^{\lambda x}-e^{\lambda a}\right)  =\lambda e^{\lambda
x}-\lambda\left(  e^{\lambda x}-e^{\lambda a}\right)  =\lambda e^{\lambda a},$
so we have
\[
0=D_{\lambda}\psi\left(  \xi\right)  =D_{\lambda}f\left(  \xi\right)
-\frac{f\left(  b\right)  -e^{\lambda\left(  b-a\right)  }f\left(  a\right)
}{e^{\lambda b}-e^{\lambda a}}\lambda e^{\lambda a}.
\]
This gives the statement in the theorem.
\end{proof}

We need the following simple lemma.

\begin{lemma}
\label{Lem31}Let $\lambda\neq0$ be a real number and $a<b.$ Then
\[
\left|  \lambda\frac{e^{\lambda a}+e^{\lambda b}}{e^{\lambda a}-e^{\lambda b}%
}\right|  \leq2\frac{e^{\left|  \lambda\right|  \left(  b-a\right)  }}{b-a}%
\]

\end{lemma}

\begin{proof}
Recall that $e^{x}-1\geq x$ for all $x\geq0.$ In the first case suppose that
$\lambda>0$. Then $\lambda\left(  b-a\right)  \geq0$ and
\[
e^{\lambda b}-e^{\lambda a}=e^{\lambda a}\left(  e^{\lambda\left(  b-a\right)
}-1\right)  \geq e^{\lambda a}\lambda\left(  b-a\right)  .
\]
Further $e^{\lambda a}\leq e^{\lambda b}$ since $\lambda>0$. Hence
\[
\lambda\frac{e^{\lambda a}+e^{\lambda b}}{e^{\lambda b}-e^{\lambda a}}%
\leq\lambda\frac{2e^{\lambda b}}{e^{\lambda a}\lambda\left(  b-a\right)
}=2\frac{e^{\lambda\left(  b-a\right)  }}{b-a}.
\]

For $\lambda<0$ we argue very similarly: $e^{\lambda a}-e^{\lambda
b}=e^{\lambda b}\left(  e^{-\lambda\left(  b-a\right)  }-1\right)  \geq
e^{\lambda b}\left|  \lambda\right|  \left(  b-a\right)  .$ Further
$e^{\lambda b}\leq e^{\lambda a}$ since $\lambda<0$ and
\[
\frac{e^{\lambda a}+e^{\lambda b}}{e^{\lambda a}-e^{\lambda b}}\leq
\frac{2e^{\lambda a}}{e^{\lambda b}\left|  \lambda\right|  \left(  b-a\right)
}=2\frac{e^{\lambda\left(  a-b\right)  }}{\left|  \lambda\right|  \left(
b-a\right)  }=2\frac{e^{\left|  \lambda\right|  \left(  b-a\right)  }}{\left|
\lambda\right|  \left(  b-a\right)  }.
\]

\end{proof}

Next we want to give an estimate of first derivative $D_{\lambda_{0}}f$ in
terms of $f$ and the second derivative $D_{\lambda_{0}\lambda_{1}}f.$ For the
case $\lambda_{0}=\lambda_{1}=0$ this is a well-known result (see
\cite[Theorem 2.4]{Schu81}) and its extension to exponential polynomials is
not difficult:

\begin{theorem}
\label{Thmoddderiv}Let $\lambda_{0},\lambda_{1}$ be real numbers and
$f:\left[  a,b\right]  \rightarrow\mathbb{C}$ be twice continuously
differentiable. Then the following estimate holds:
\[
\left|  D_{\lambda_{0}}f\left(  a\right)  \right|  \leq4\frac{e^{\left(
\left|  \lambda_{0}\right|  +\left|  \lambda_{1}\right|  \right)  \left(
b-a\right)  }}{b-a}\max\left\{  \left|  f\left(  a\right)  \right|  ,\left|
f\left(  b\right)  \right|  \right\}  +2\max_{t\in\left[  a,b\right]  }\left|
D_{\lambda_{1}}D_{\lambda_{0}}f\left(  t\right)  \right|  \left(  b-a\right)
e^{\left|  \lambda_{1}\right|  \left(  b-a\right)  }.
\]

\end{theorem}

\begin{proof}
Assume at first that $\lambda_{0}$ and $\lambda_{1}$ are not zero and that $f$
is real-valued. We apply Theorem \ref{ThmIntermed} to $f$ and $\lambda
=\lambda_{0},$ so there exists $\xi_{1}\in\left(  a,b\right)  $ such that
\[
\frac{e^{\lambda_{0}a}f\left(  b\right)  -e^{\lambda_{0}b}f\left(  a\right)
}{e^{\lambda_{0}b}-e^{\lambda_{0}a}}\lambda_{0}=D_{\lambda_{0}}f\,\left(
\xi_{1}\right)  .
\]
Then Lemma \ref{Lem31} yields the estimate
\begin{align*}
\left\vert D_{\lambda_{0}}f\,\left(  \xi_{1}\right)  \right\vert  &  \leq
\frac{e^{\lambda_{0}a}+e^{\lambda_{0}b}}{\left\vert e^{\lambda_{0}%
b}-e^{\lambda_{0}a}\right\vert }\left\vert \lambda_{0}\right\vert \max\left\{
\left\vert f\left(  a\right)  \right\vert ,\left\vert f\left(  b\right)
\right\vert \right\} \\
&  \leq2\frac{e^{\left\vert \lambda_{0}\right\vert \left(  b-a\right)  }}%
{b-a}\max\left\{  \left\vert f\left(  a\right)  \right\vert ,\left\vert
f\left(  b\right)  \right\vert \right\}  .
\end{align*}
Now apply Theorem \ref{ThmIntermed} to the interval $\left[  a,\xi_{1}\right]
$ for $\lambda=\lambda_{1}$ and the function $D_{\lambda_{0}}f.$ Then there
exists $\xi_{2}\in\left(  a,\xi_{1}\right)  \subset\left(  a,b\right)  $ such
that
\[
\frac{e^{\lambda_{1}a}D_{\lambda_{0}}f\left(  \xi_{1}\right)  -e^{\lambda
_{1}\xi_{1}}D_{\lambda_{0}}f\left(  a\right)  }{e^{\lambda_{1}\xi_{1}%
}-e^{\lambda_{1}a}}\lambda_{1}=D_{\lambda_{1}}D_{\lambda_{0}}f\left(  \xi
_{2}\right)  .
\]
Thus $e^{\lambda_{1}\xi_{1}}D_{\lambda_{0}}f\left(  a\right)  =e^{\lambda
_{1}a}D_{\lambda_{0}}f\left(  \xi_{1}\right)  -\left(  e^{\lambda_{1}\xi_{1}%
}-e^{\lambda_{1}a}\right)  /\lambda_{1}\times D_{\lambda_{1}}D_{\lambda_{0}%
}f\left(  \xi_{2}\right)  $ and
\[
\left\vert D_{\lambda_{0}}f\left(  a\right)  \right\vert \leq e^{\lambda
_{1}\left(  a-\xi_{1}\right)  }\left\vert D_{\lambda_{0}}f\left(  \xi
_{1}\right)  \right\vert +\left\vert D_{\lambda_{1}}D_{\lambda_{0}}f\left(
\xi_{2}\right)  \right\vert \frac{\left\vert 1-e^{\lambda_{1}\left(  a-\xi
_{1}\right)  }\right\vert }{\left\vert \lambda_{1}\right\vert }.
\]
We estimate
\begin{align*}
\left\vert \frac{e^{\lambda_{1}\left(  a-\xi_{1}\right)  }-1}{\lambda_{1}%
}\right\vert  &  \leq\sum_{n=1}^{\infty}\frac{\left\vert \lambda_{1}\left(
a-\xi_{1}\right)  \right\vert ^{n}}{\left\vert \lambda_{1}\right\vert n!}%
\leq\sum_{n=1}^{\infty}\frac{\left\vert \lambda_{1}\right\vert ^{n-1}}%
{n!}\left(  b-a\right)  ^{n}\\
&  \leq\left(  b-a\right)  \sum_{n=1}^{\infty}\frac{\left\vert \lambda
_{1}\right\vert ^{n-1}}{\left(  n-1\right)  !}\left(  b-a\right)
^{n-1}=\left(  b-a\right)  e^{\left\vert \lambda_{1}\right\vert \left(
b-a\right)  }.
\end{align*}
Thus
\[
\left\vert D_{\lambda_{0}}f\left(  a\right)  \right\vert \leq2e^{\left\vert
\lambda_{1}\right\vert \left(  b-a\right)  }\frac{e^{\left\vert \lambda
_{0}\right\vert \left(  b-a\right)  }}{b-a}\max\left\{  \left\vert f\left(
a\right)  \right\vert ,\left\vert f\left(  b\right)  \right\vert \right\}
+\left\vert D_{\lambda_{1}}D_{\lambda_{0}}f\left(  \xi_{2}\right)  \right\vert
\left(  b-a\right)  e^{\left\vert \lambda_{1}\right\vert \left(  b-a\right)
}.
\]
Next consider the case that $f$ is complex-valued. Decompose $f$ into its real
and imaginary part, say $f=f_{1}+if_{2}$ with real-valued functions
$f_{1},f_{2}.$ Then
\begin{equation}
\left\vert D_{\lambda_{0}}f\left(  a\right)  \right\vert \leq\left\vert
D_{\lambda_{0}}f_{1}\left(  a\right)  \right\vert +\left\vert D_{\lambda_{0}%
}f_{2}\left(  a\right)  \right\vert . \label{eqtwos}%
\end{equation}
Now estimate each summand in (\ref{eqtwos}) as above. Since $\left\vert
f_{j}\left(  t\right)  \right\vert \leq\left\vert f\left(  t\right)
\right\vert $ and $\left\vert D_{\lambda_{1}}D_{\lambda_{0}}f_{j}\left(
t\right)  \right\vert \leq\left\vert D_{\lambda_{1}}D_{\lambda_{0}}f\left(
t\right)  \right\vert $ (note that $D_{\lambda_{1}}D_{\lambda_{0}}f_{1}$ is
the real part of $D_{\lambda_{1}}D_{\lambda_{0}}f),$ we obtain the estimate
with the factor $2.$ The case of arbitrary real numbers $\lambda_{0}$ and
$\lambda_{1}$ follows from a continuity argument.
\end{proof}

\begin{theorem}
\label{ThmApp}Let $\lambda_{m},m\in\mathbb{N}_{0}$ be real numbers and $f\in
C^{\infty}\left[  x_{0},x_{0}+\gamma\right]  $ with $\gamma>0.$ Assume that
$0<2\delta<\gamma.$ Then the estimate
\begin{align*}
\left\vert D^{\left(  2m+1\right)  }f\left(  x\right)  \right\vert  &
\leq2\max\left\{  \frac{2}{\delta},\delta\right\}  e^{\left(  \left\vert
\lambda_{2m}\right\vert +\left\vert \lambda_{2m+1}\right\vert \right)  \delta
}\\
&  \left(  \max_{t\in\left[  x_{0},x_{0}+2\delta\right]  }\left\vert
D^{\left(  2m\right)  }f\left(  t\right)  \right\vert +\max_{t\in\left[
x_{0},x_{0}+2\delta\right]  }\left\vert D^{\left(  2m+2\right)  }f\left(
t\right)  \right\vert \right)
\end{align*}
holds for all $x\in\left[  x_{0},x_{0}+\delta\right]  .$
\end{theorem}

\begin{proof}
We use the estimate in Theorem \ref{Thmoddderiv}
\[
\left\vert D_{\lambda_{0}}f\left(  a\right)  \right\vert \leq4e^{\left(
\left\vert \lambda_{0}\right\vert +\left\vert \lambda_{1}\right\vert \right)
\left(  b-a\right)  }\frac{\max\left\{  \left\vert f\left(  a\right)
\right\vert ,\left\vert f\left(  b\right)  \right\vert \right\}  }{b-a}%
+2\max_{t\in\left[  a,b\right]  }\left\vert D_{\lambda_{1}}D_{\lambda_{0}%
}f\left(  t\right)  \right\vert \left(  b-a\right)  e^{\left\vert \lambda
_{1}\right\vert \left(  b-a\right)  }%
\]
for the function
\[
D^{\left(  2m\right)  }f\left(  t\right)  =\left(  \frac{d}{dt}-\lambda
_{0}\right)  \cdots\left(  \frac{d}{dt}-\lambda_{2m-1}\right)  f\left(
t\right)
\]
for the exponents $\lambda_{0}:=\lambda_{2m}$ and $\lambda_{1}:=\lambda
_{2m+1}$, and for $a=x$ and $b:=x+\delta$ where it is assumed that
$x\in\left[  x_{0},x_{0}+\delta\right]  $. It follows that $t\in\left[
x,x+\delta\right]  $ is in $\left[  x_{0},x_{0}+2\delta\right]  .$ Then
\begin{align*}
\left\vert \left(  \frac{d}{dx}-\lambda_{2m}\right)  D^{\left(  2m\right)
}f\left(  x\right)  \right\vert  &  \leq\frac{4}{\delta}e^{\left(  \left\vert
\lambda_{2m}\right\vert +\left\vert \lambda_{2m+1}\right\vert \right)  \delta
}\max_{t\in\left[  x_{0},x_{0}+2\delta\right]  }\left\vert D^{\left(
2m\right)  }f\left(  t\right)  \right\vert \\
&  +2\delta\max_{t\in\left[  x_{0},x_{0}+2\delta\right]  }\left\vert
D^{\left(  2m+2\right)  }f\left(  t\right)  \right\vert e^{\left\vert
\lambda_{2m+1}\right\vert \delta}.
\end{align*}
Using the trivial inequality $e^{\left\vert \lambda_{2m+1}\right\vert \delta
}\leq e^{\left(  \left\vert \lambda_{2m}\right\vert +\left\vert \lambda
_{2m+1}\right\vert \right)  \delta}$ the statement follows.
\end{proof}

\end{document}